\documentclass{amsart}
\usepackage{amsmath}
\usepackage{amssymb}
\usepackage{amsthm}
\usepackage{amsfonts}
\usepackage{tikz-cd}
\usepackage{enumitem}
\usepackage[hidelinks]{hyperref}
\usepackage{dsfont}
\usepackage{mathrsfs}
\usepackage{graphicx}

\makeatletter
\@namedef{subjclassname@2020}{\textup{2020} Mathematics Subject Classification}
\makeatother

\newtheorem{theorem}{Theorem}[section]
\newtheorem{proposition}[theorem]{Proposition}
\newtheorem{corollary}[theorem]{Corollary}
\newtheorem{lemma}[theorem]{Lemma}

\theoremstyle{definition}
\newtheorem{example}[theorem]{Example}
\newtheorem{notation}[theorem]{Notation}
\newtheorem{definition}[theorem]{Definition}
\newtheorem{setup}[theorem]{Setup}
\newtheorem{remark}[theorem]{Remark}

\theoremstyle{remark}

\DeclareMathOperator{\Mon}{\mathrm{Mon}}
\newcommand{\Spec}{\operatorname{Spec}}
\def\C{\mathbb{C}}

\def\R{\mathbb{R}}
\def\H{\mathsf{H}}
\def\L{\mathrm{L}}

\def\P{\mathbb{P}}
\def\A{\mathbb{A}}

\def\H{\mathrm{H}}
\def\F{\mathrm{F}}
\def\S{\mathrm{S}}

\def\Gm{\mathbb{G}_m}
\def\GL{\mathrm{GL}}
\def\PGL{\mathrm{PGL}}
\def\SL{\mathrm{SL}}
\def\ASL{\mathrm{ASL}}
\def\AGL{\mathrm{AGL}}

\def\V{\mathrm{V}}

\def\cO{\mathcal{O}}

\def\cX{\mathcal{X}}
\def\cY{\mathcal{Y}}
\def\cZ{\mathcal{Z}}

\def\cW{\mathcal{W}}
\def\cL{\mathcal{L}}
\def\cV{\mathcal{V}}
\def\cU{\mathcal{U}}

\def\cM{\mathcal{M}}
\def\cF{\mathcal{F}}
\def\cG{\mathcal{G}}
\def\cS{\mathcal{S}}
\def\mC{\mathrm{C}}
\def\G{\mathrm{G}}
\def\M{\mathrm{M}}
\def\W{\mathrm{W}}
\def\K{\mathrm{K}}
\def\fS{\mathfrak{S}}

\def\Aut{\mathrm{Aut}}
\def\uAut{\underline{\mathrm{Aut}}}
\def\Inn{\mathrm{Inn}}

\def\Zen{\mathrm{Z}}
\def\Out{\mathrm{Out}}

\def\Gal{\mathrm{Gal}}
\def\Sch{\mathrm{Sch}}
\def\Ima{\mathrm{Image}}

\def\Nor{\mathrm{N}}

\begin{document}
	\title{Stacks, Monodromy and Symmetric Cubic Surfaces}
	\author{ALBERTO LANDI}
	\address{Brown University, 151 Thayer Street, Providence, RI 02912, USA}
	\email{alberto\_landi@brown.edu}
	
	
	
	\begin{abstract}
		We investigate monodromy groups arising in enumerative geometry, with a particular focus on how these groups are influenced by prescribed symmetries. To study these phenomena effectively, we work in the framework of moduli stacks rather than moduli spaces. This perspective proves broadly useful for understanding and constructing monodromy. We illustrate these ideas through several examples, with special attention to the 27 lines on a cubic surface, assuming the surface admits a given symmetry group.
	\end{abstract}
	
	\maketitle
	
	\section{Introduction}
	An enumerative problem induces a generically finite and generically étale cover $\F:Y\rightarrow X$, where $Y$ is the moduli space of solutions and $X$ the moduli space of geometric conditions, with the classical solution count given by the degree of $\F$. Then, one can ask what is the monodromy group of $\F$, also known as the Galois group of the enumerative problem. This was first studied by Jordan in~\cite{Jor70}, and about a hundred years later the field enjoyed a revival with the work of Harris~\cite{Har79}. Many subsequent computations have been carried out, studying for instance Fano problems~(\cite{HK22}) or Schubert problems on Grassmannians~(\cite{Vak06}).
	See~\cite{SY21} for a more complete introduction on Galois groups of enumerative problems.
	
	In this paper, we show how stacks can be a powerful tool to compute monodromy of enumerative problems, replacing the classical cover $F$ with a cover $\cF:\cY\rightarrow\cX$ between moduli stacks. In particular, the stack-theoretic stabilizer groups can be used to explicitly construct elements of the monodromy group, and in some cases, generate it entirely.
	As an example, in Section~\ref{subsec: first examples} we give an alternative computation of the Galois group of two classical enumerative problems.
	
	Our techniques are particularly effective when studying how symmetries interact with monodromy, which is a continuation of the program initiated by Brazelton and Raman in~\cite{BR24}. As an example, let $\cM$ be the moduli stack of complex smooth cubic surfaces, with coarse moduli space $\pi_{\cM}:\cM\rightarrow\M$. The enumerative problem of the 27 lines yields a finite étale cover $\cF:\cL\rightarrow\cM$ of degree 27, inducing a generically étale cover $\F:\L\rightarrow\M$ between coarse moduli spaces. If $\mC_2$ is the group with two elements, we denote by $\cM^{=\mC_2}$ the locally-closed substack of $\cM$ of cubic surfaces $S$ such that $\Aut(S)\simeq\mC_2$, and by $\M^{=\mC_2}$ its coarse moduli space.
	\begin{theorem}\label{thm: introduction}
		The monodromy groups of the covers $\cF^{=\mC_2}:\cL|_{\cM^{=\mC_2}}\rightarrow\cM^{=\mC_2}$ and $\F^{=\mC_2}:\L|_{\M^{=\mC_2}}\rightarrow\M^{=\mC_2}$ are
		\begin{align*}
			&\Mon_{\cF^{=\mC_2}}\simeq\mathrm{GO}_4^+(3),\\
			&\Mon_{\F^{=\mC_2}}\simeq\mathrm{PGO}_4^+(3),
		\end{align*}
		where $\mathrm{GO}_4^+(3)$ is the stabilizer in $W(E_6)$ of a tritangent plane, and $\mathrm{PGO}_4^+(3)$ the quotient by its center.
	\end{theorem}	
	This is proved in Theorem~\ref{thm: mu2 monodromy}. We obtained analogous explicit results for $\G$ equal to $\fS_4$, $\fS_3$ and $\fS_3\times\mC_2$, where $\fS_n$ is the symmetric group of $n$ letters; see Theorems~\ref{thm: S4 monodromy},~\ref{thm: S3 monodromy}, and~\ref{thm: S3xmu2 monodromy}. In particular, we recover the result~\cite[Theorem 1.2]{BR24}. However, we remark that their Hodge-theoretic techniques have other advantages, for example a description of the moduli of $\fS_4$-symmetric cubic surfaces as an arithmetic ball quotient, see~\cite[Theorem 1.1]{BR24}.
	\begin{remark}
		The difference between the two monodromy groups in Theorem~\ref{thm: introduction} reflects the fact that the moduli map $\pi^{=\mC_2}:\cM^{=\mC_2}\rightarrow\M^{=\mC_2}$ is a $\mC_2$-gerbe, and that $\Aut(S)\simeq\mC_2$ acts faithfully on the lines on $S$. Consequently, passing to coarse moduli spaces discards the monodromy contribution arising from the generic stabilizer, information which is retained on the stack.
		Moreover, the central extension
		\[
		\begin{tikzcd}
			0\arrow[r] & \mC_2\arrow[r] & \Mon_{\cF^{=\mC_2}}\arrow[r] & \Mon_{\F^{=\mC_2}}\arrow[r] & 0
		\end{tikzcd}\]
		does not split, implying that the $\mC_2$-gerbe $\pi^{=\mC_2}$ is non-trivial. This highlights how the study of monodromy groups provides information on the underlying gerbes.
		
		The exact sequence above generalizes to any symmetric enumerative problem, providing a strong constraint on the monodromy group by relating it to the Galois group of the original cover. See Corollary~\ref{cor: inclusion in normalizer} for the precise statement. We expect this to be the only constraint in many situations, which would lead to a general formula for solving such problems.
	\end{remark}
	
	When considering restrictions of covers to loci corresponding to objects carrying some symmetries as above, stacks not only represent a useful tool, but also the right framework to work in. For instance, the subscheme $\M^{\mC_2}$ of $\M$ parametrizing surfaces $S$ with $\mC_2\subseteq\Aut(S)$ is not a normal scheme, hence there can be covers of it that are connected but reducible; this is the case for the induced cover $\F^{\mC_2}$. In particular, the monodromy group of the restriction of $\F^{\mC_2}$ to an open subscheme of $\M^{\mC_2}$ over which $\F^{\mC_2}$ is étale can depend on the open chosen, contrarily to what happens in the classical case, see~\cite[Section 1]{Har79}. There are several solutions to this problem, and they seem to require the use of stacks to get a satisfying answer. These are discussed in Example~\ref{exm: G cubic surfaces} and throughout the paper. Notice that $\cF^{\mC_2}$ is étale, hence one could have considered the Galois group of the cover without restricting it to $\cM^{=\mC_2}$. However, it turns out that $\Mon_{\cF^{\mC_2}}=\Mon_{\cF}=W(E_6)$, see Example~\ref{exm: G cubic surfaces}, hence probably too big to be meaningful.
	
	One of the tools we use to compute and make sense of the monodromy in situations such as the above is the introduction of the stack of equivariant objects. Continuing the example above, we consider the stack $\cM_{\mC_2}$ whose objects are pairs $(S,\rho:\mC_2\hookrightarrow\Aut(S))$, where $S$ is a smooth cubic surface and $\rho$ a faithful action of $\mC_2$ on $S$. The stack $\cM_{\mC_2}$ is smooth but highly disconnected, see Proposition~\ref{prop: basic facts about X{G}}. We study its connected components, and the maps $\cM_{\mC_2}\rightarrow\cM^{\mC_2}\rightarrow\M^{\mC_2}$, where $\cM^{\mC_2}\subset\cM$ parametrizes smooth cubics surfaces $S$ with $\mC_2\subseteq\Aut(S)$. The preimage $\cM_{=\mC_2}$ of $\cM^{=\mC_2}$ is connected, and in Lemma~\ref{lem: representation mu2} we characterize its closure $\overline{\cM}_{=\mC_2}$ in $\cM_{\mC_2}$. Then, $\overline{\cM}_{=\mC_2}\rightarrow\M^{\mC_2}$ is a natural resolution of singularities in the category of stacks, and it restricts to an isomorphism $\cM_{=\mC_2}\simeq\cM^{=\mC_2}$. This allows us to obtain $\Mon_{\cF^{=\mC_2}}$ by computing the Galois group of the cover over $\overline{\cM}_{=\mC_2}$. This is worked out in broader generality in Section~\ref{sec: G cubic surfaces}.
	\subsection*{Structure of the Paper}
	In Section~\ref{sec: fundamental group}, we begin by recalling the definition and key properties of the fundamental group for stacks, particularly in Subsections~\ref{subsec: preliminaries on fundamental group} and~\ref{subsec: properties of monodromy}. As an application, Subsection~\ref{subsec: first examples} provides an alternative computation of the monodromy group for the classical enumerative problems of the 27 lines on a cubic surface and the 9 flexes on a plane cubic.
	In the subsequent sections, we develop tools for working with enumerative problems in the presence of prescribed symmetries.
	Specifically, in Section~\ref{sec: G cubic surfaces} we define the stack parametrizing objects with faithful $\G$-actions, explore its properties, and analyze its connected components in the case of cubic surfaces.
	Finally, in Section~\ref{sec: monodromy G cubic surfaces}, we compute the monodromy group for the problem of the 27 lines on a cubic surface, with the additional requirement that the surfaces admit a specified group of symmetries $\G$ between $\fS_4$, $\fS_3$, $\fS_3\times\mC_2$ and $\mC_2$.
	\subsection*{Acknowledgments}
	I am sincerely grateful to Brendan Hassett for his continuous support, valuable guidance, and for encouraging me to pursue this project. I would also like to thank Thomas Brazelton for stimulating conversations on this and related topics, and Sidhanth Raman, whose presentation at Brown University on their joint work was an important source of inspiration for this paper. Finally, I am especially thankful to Dan Abramovich for reading a draft of the paper and offering his insightful feedback.
	
	\section{Fundamental Group and Monodromy of Algebraic Stacks}\label{sec: fundamental group}
	In this section we recall the definition of the fundamental group of an algebraic stack and its main properties. In particular, we study the notion of monodromy of a representable finite étale cover in the general context of algebraic stacks. In the last subsection, we give some first examples where we can apply the general theory to compute monodromy groups of some enumerative problems.
	
	\subsection{Preliminaries on Fundamental Groups of Stacks}\label{subsec: preliminaries on fundamental group}
	The constructions and results of this subsection are mainly due to Behrang Noohi, who extended Grothendieck's definition (\cite{SGA1}) of the étale fundamental group of a pointed scheme to the case of stacks; see~\cite{Noo04}. See~\cite{Zoo01} for an alternative approach. We will be mainly interested in Deligne-Mumford stacks, but the theory works in general, so we will not restrict to that particular class of stacks in this section.
	
	\begin{notation}
		Throughout the section, $\cX$ and $\cY$ will be algebraic stacks, and $\cX$ will be assumed to be connected. Given a geometric point $x:\Spec K\rightarrow\cX$ of $\cX$, the stabilizer group $\G_x$ (or automorphism group) of $\cX$ at $x$ is the fiber product
		\[
		\begin{tikzcd}
			\G_x\arrow[d]\arrow[r] & \Spec K\arrow[d,"{(x,x)}"]\\
			\cX\arrow[r,"\Delta"] & \cX\times\cX.
		\end{tikzcd}
		\]
		Recall that every morphism of stacks is a functor between the underlying categories. In particular, one can consider transformations between two morphism, which are simply called 2-morphisms.
	\end{notation}
	The main difference with the case of schemes is the presence of `hidden loops', which are topologically trivial loops induced by the automorphism groups of points. If the stack of interest is a moduli stack parametrizing families of some geometric object, the hidden paths correspond to isotrivial families.
	
	\begin{definition}[{\cite[Definitions 3.1, 3.3]{Noo04}}]
		Let $x,x':\Spec K\rightarrow\cX$ be two geometric points. A hidden path from $x$ to $x'$ is a 2-morphism from $x$ to $x'$. The hidden fundamental group $\pi_1^h(\cX,x)$ of $\cX$ at $x$ is the group of self-transformations of $x$, that is, the group of 2-automorphisms of $x$. The hidden fundamental groupoid $\Pi_1^h(\cX)$ of $\cX$ is the groupoid with objects the geometric points of $\cX$ and with arrows the 2-morphisms between them.
		
		A pointed map is a pair $(f,\phi):(\cY,y)\rightarrow(\cX,x)$ where $f:\cY\rightarrow\cX$ is a morphism and $\phi$ is an hidden path from $x$ to $f(y)$. A (pointed) covering map is a (pointed) representable, finite, étale morphism from a connected scheme. Usually, we denote a pointed map simply by $f$, instead of the whole pair $(f,\phi)$.
	\end{definition}
	\begin{remark}\label{rmk: group scheme structure of hidden fundamental group}
		Notice that every pointed map $(f,\phi):(\cY,y)\rightarrow(\cX,x)$ induces a morphism between hidden fundamental groups. Moreover, it is possible to identify $\pi_1^h(\cX,x)$ with the $K$-points of the stabilizer group $\G_x$ of $\cX$ at the $K$-point $x$, thus endowing $\pi_1^h(\cX,x)$ with a group scheme structure; see~\cite[Proposition 3.4]{Noo04}.
	\end{remark}
	Given a pointed connected stack $(\cX,x)$, Noohi defines the fundamental group $\pi_1(\cX,x)$ as follows. First, he considers the Galois category of representable, finite, étale covers of $\cX$ by algebraic stacks, and the fiber functor $F_x$ sending a cover to its fiber over $x$, similarly to the classical case of schemes. More precisely, if $f:\cY\rightarrow\cX$ is a covering map, then $F_x(\cY)$ is the set of isomorphism classes of pairs $(y,\phi)$, where $y$ is a geometric point of $\cY$ and $\phi$ an hidden path from $x$ to $f(y)$; two pairs $(y,\phi)$ and $(y',\phi')$ are considered isomorphic if there is an hidden path from $y$ to $y'$ compatible with $\phi$ and $\phi'$. The main difference with the case of schemes is the necessity to include possible 2-morphisms between geometric points into the definition; see~\cite[Section 4]{Noo04} for details and the rigorous definitions. Then, the fundamental group $\pi_1(\cX,x)$ of $\cX$ at $x$ is defined as the group of self-transformations of the fiber functor $F_x$. Finally, the fundamental groupoid $\Pi_1(\cX)$ is the groupoid whose objects are geometric points of $\cX$ and the arrows are transformations of fiber functors $F_x\rightarrow F_{x'}$.
	The construction is covariant with respect to pointed maps.
	
	\begin{remark}\label{rmk: fundamental correspondence}
		It is again true that $\pi_1(\cX,x)$ is a profinite group, and the correspondence between its open subgroups and isomorphisms classes of pointed covers still holds. Given a pointed covering map $(\cY,y)\rightarrow(\cX,x)$, the associated open subgroup of $\pi_1(\cX,x)$ is the image of $\pi_1(\cY,y)\rightarrow\pi_1(\cX,x)$.
	\end{remark}
	Remark~\ref{rmk: fundamental correspondence} allows us to define the monodromy group of a cover.
	\begin{definition}\label{def: monodromy for stacks}
		Let $f:(\cY,y)\rightarrow(\cX,x)$ a connected pointed covering map, and let $H_f$ be the associated open subgroup of $\pi_1(\cX,x)$, that is, the image under the pushforward map
		\[
		\begin{tikzcd}
			\pi_1(\cY,y)\arrow[r] & \pi_1(\cX,x).
		\end{tikzcd}
		\]
		Let $N_f$ be the largest normal subgroup of $\pi_1(\cX,x)$ contained in $H_f$. This is the finite intersection of conjugates of $H_f$, hence open and of finite index in $\pi_1(\cX,x)$. We define the monodromy group $\Mon_f$ of $f$ as the quotient
		\begin{equation}
			\Mon_f:=\pi_1(\cX,x)/N_f.
		\end{equation}
		Moreover, we define the Galois closure of $f$ to be the pointed covering map
		\[
		\begin{tikzcd}
			\widetilde{f}:(\widetilde{Y},\widetilde{y})\arrow[r] & (\cX,x)
		\end{tikzcd}
		\]
		corresponding to $N_f$. If $\widetilde{\cY}=\cY$, we say that $f$ is a Galois cover with Galois group $\Mon_f$.
	\end{definition}
	\begin{remark}\label{rmk: Galois covers are torsors}
		A Galois cover $f$ is a $\Mon_f$-torsor, see~\cite[\href{https://stacks.math.columbia.edu/tag/03SF}{Tag 03SF}]{Sta25}.
	\end{remark}
	\begin{remark}
		The definition given here is too abstract to allow computations of monodromy groups of covers. One better way of visualizing the monodromy group is through its natural action on the fiber of the covering, which however can be tricky in the case of stacks. Indeed, contrary to the case of schemes, if $f:Y\rightarrow X$ is a Galois cover with Galois group $G$, then the action of $G$ on $Y$ is not necessarily faithful, as Example~\ref{exm: fundamental group classifying stack} shows; however, it is clearly faithful on the scheme-theoretic fibers. In Subsection~\ref{subsec: properties of monodromy}, we give a more concrete interpretation of the monodromy group.
	\end{remark}
	A key feature of the fundamental group for stacks is the existence for every geometric point $x$ in $\cX$ of a natural morphism
	\begin{equation}\label{eq: map omegax}
		\begin{tikzcd}
			\omega_x:\pi_1^h(\cX,x)\arrow[r] & \pi_1(\cX,x),
		\end{tikzcd}
	\end{equation}
	which comes from a morphism of groupoids $\Pi_1^h(\cX)\rightarrow\Pi_1(\cX)$. The map~\eqref{eq: map omegax} is defined by mapping each hidden loop $\gamma$ at $x$ to the self-transformation $F_x\rightarrow F_x$ sending a pair $(y,\phi)$ to $(y,\gamma^{-1}\phi)$; see~\cite[Section 4]{Noo04} for the details.
	\begin{example}[{\cite[Example 4.2]{Noo04}}]\label{exm: fundamental group classifying stack}
		Let $\cX=B\G=[\Spec k/G]$ be the classifying stack of an algebraic group $\G$ of finite type over an algebraically closed field $k$, with geometric point the projection $x:\Spec k\rightarrow B\G$. Then, the hidden fundamental group is $\pi_1^h(B\G,x)\simeq\G$, while
		\begin{equation}
			\Ima(\omega_x)=\pi_1(B\G,x)=\G/\G^o,
		\end{equation}
		where $\G^o$ is the connected component of the identity in $\G$. In particular, if $\G$ is a finite étale group scheme over $k$, then $\pi_1(B\G)\simeq\G$, and $\Spec k\rightarrow B\G$ is the universal cover. In this case, the monodromy group of the universal cover $\Spec k\rightarrow B\G$ is $\G$, and it acts trivially on the total space.
	\end{example}
	\begin{remark}\label{rmk: finite image omegax}
		Suppose $\cX$ is quasi-separated and of finite type over a field $k$. Then, the above example shows that for every geometric point $x$ of $\cX$ the image of $\omega_x$ is finite, as that map factors through the fundamental group of the residue gerbe $B\G_x$ at $x$, which the example showed to be finite.
	\end{remark}
	
	The map $\omega_x$ will be one of our main tools to generate elements in the monodromy group of a fixed covering. In order to do so effectively, we need to understand what are the source and kernel of $\omega_x$.	
	We know that the source is $\pi_1^h(\cX,x)\simeq\G_x$, the stabilizer group of $\cX$ at $x$, see Remark~\ref{rmk: group scheme structure of hidden fundamental group}. In practice, we are usually able to show directly that $\G_x$ acts faithfully on the fibers of some cover $f$, so that we get the stronger claim that $\G_x$ injects into $\Mon_f$.
	\begin{remark}
		In~\cite[Corollary 5.4, Theorem 7.3]{Noo04}, the author gives both local and global criteria for when the map $\omega_x$ is injective. When $\cX$ is a Deligne-Mumford stack, this holds if and only if there exists a covering map $f:\cY\rightarrow\cX$ that is a uniformization locally around $x$, that is, $\cY$ is a scheme over $x$. Moreover, $\cX$ admits a global uniformization if and only if $\omega_x$ is injective for every $x$.
	\end{remark}
	
	Now, we are interested in comparing the fundamental group of an algebraic stack $\cX$ with the fundamental group of its moduli space $X$, intended in the broad sense as in~\cite[Definition 7.1]{Noo04}, whenever this exists.
	
	Suppose given a connected Deligne-Mumford stack $\cX$ with a coarse moduli space $\pi_{\cX}:\cX\rightarrow X$, and let $g:Y\rightarrow X$ be a finite, étale cover between connected algebraic spaces. Then, one can construct the fiber product
	\[
	\begin{tikzcd}
		\cY\arrow[d,"f"]\arrow[r,"\pi_{\cY}"] & Y\arrow[d,"g"]\\
		\cX\arrow[r,"\pi_{\cX}"] & X.
	\end{tikzcd}
	\]
	We know that $f$ is representable, finite, étale of the same degree as $g$, $\pi_{\cY}$ is a coarse moduli space, and $\cY$ is connected as $\pi_{\cY}$ is an homeomorphism. Notice that the monodromy group $\Mon_g$ of $g$ acts on the fibers of $f$ in a natural way. If $Y\rightarrow X$ is a Galois cover with Galois group $\G$ then the same holds for $f$, and the action of $G$ on $\cY$ is free. In particular, $\Mon_f=\Mon_g$.
	
	In general, not every cover $f:\cY\rightarrow\cX$ is of this form. For instance, for every finite étale group $\G$ over $k$, the cover $\Spec k\rightarrow B\G$ is Galois with group $\G$, but $\G$ acts trivially on the total space; see Example~\ref{exm: fundamental group classifying stack}. The point is that in this case $f$ does not induce an isomorphism between the stabilizers; it turns out that this is the only obstruction.
	\begin{definition}[{\cite[Definition 7.5]{Noo04}}]\label{def: FPR covers}
		We say that a representable, finite, étale cover $f:\cY\rightarrow\cX$ is fixed point reflecting, or FPR, if it induces an isomorphism between the geometric stabilizers of every corresponding geometric point.
	\end{definition}
	The following Proposition is a direct consequence of~{\cite[Proposition 7.8, Theorem 7.11]{Noo04}}.
	\begin{proposition}\label{prop: correspondence FPR}
		Let $\cX$ be a connected algebraic stack with a moduli space $\pi_{\cX}:\cX\rightarrow X$. There is a natural equivalence between the Galois category of FPR covers of $\cX$ and the Galois category of covers of its moduli space $X$, sending an FPR cover of $\cX$ to the induced map between moduli spaces, with inverse given by pullback along $\pi_{\cX}$. This correspondence sends connected (respectively, Galois) covers to connected (respectively, Galois) covers, and induces an isomorphisms between the corresponding monodromy groups.
	\end{proposition}
	
	Morally, we might expect that every cover of an algebraic stack $\cX$ is a combination of FPR covers and covers of classifying stacks; this is true in some precise sense, as we will see in the next subsection.
	
	In the following case, it is particularly easy to compare the fundamental group of $\cX$ with that of its moduli space, which is presented in~\cite{Noo04} with the language of monotonous gerbes; see~\cite[Definition 9.1]{Noo04} for the definition.
	Recall that a fppf gerbe $\pi:\cX\rightarrow X$ is equivalently a smooth morphism admitting an flat, finitely-represented, surjective morphism $S\rightarrow X$ such that the pullback of $\pi$ becomes $\cX\times_XS\simeq B\G\rightarrow S$ for some fppf group scheme $G\rightarrow S$.
	\begin{proposition}[{\cite[Propositions 10.3, 10.4, Corollary 10.5]{Noo04}}]\label{prop: monotonous gerbe case}
		Let $\cX$ be a connected algebraic stacks whose moduli map $\pi_{\cX}:\cX\rightarrow X$ is an fppf gerbe, and let $f:\cY\rightarrow\cX$ be a connected covering map. Then, also $\cY$ is an fppf gerbe over its moduli space, and $f$ is FPR if and only if it is FPR at some geometric point $y$ of $\cY$. Moreover, for every geometric point $x$ of $\cX$, the image of $\omega_x$ is normal in $\pi_1(\cX,x)$, and its quotient is the fundamental group of $X$.
	\end{proposition}
	\begin{remark}\label{rmk: generic gerbe}
		By generic flatness, every reduced Noetherian algebraic stack admits an open substack of the form above. See~\cite[Sections 6.4.2, 6.4.3]{Alp25} for technical facts on gerbes.
	\end{remark}
	We conclude the subsection by extending the definition of monodromy group to covers that are not necessarily connected.
	\begin{definition}\label{def: monodromy for non connected covers}
		Let $f:\cY\rightarrow\cX$ be a representable, finite, étale cover
		of an irreducible stack $\cX$, with $\cY$ not necessarily connected. Let $\cY_1,\ldots,\cY_r$ be the connected components, with induced maps $f_i:\cY_i\rightarrow\cX$. Let $x$ be a geometric point of $\cX$, and for every $i$ fix a geometric point $y_i$ such that $f_i(y_i)=x_i$. Recall that $N_{f_i}$ is the largest normal subgroup of $\pi_1(\cX,x)$ contained in the image of $f_{i*}:\pi_1(\cY_i,y_i)\rightarrow\pi_1(\cX,x)$. Let $N_f=\cap_i N_{f_i}$. We define the monodromy group $\Mon_f$ of $f$ as
		\[
		\Mon_f=\pi_1(\cX,x)/N_f.
		\]
	\end{definition}
	\begin{remark}
		The subgroup $N_f$ of $\pi_1(\cX,x)$ is clearly normal, so the quotient makes sense. Notice that $N_f$ can be equivalently described as the largest normal subgroup contained in the image of every $f_{i*}$.
	\end{remark}
	\subsection{Properties of Monodromy}\label{subsec: properties of monodromy}
	In this subsection, we show that the monodromy group of a cover between irreducible stacks can be computed after restricting the cover to any non-empty open substack. To do so, we give an equivalent definition of monodromy for algebraic stacks, that is well known in the case of schemes.
	In this subsection, we will always assume $\cX$ to be irreducible.
	
	Let $f:\cY\rightarrow\cX$ be a representable, finite, étale cover of degree $n$, between irreducible stacks. Let $\cY^{(n)}$ be the complement of the big diagonal in the $n$-fold product $\cY\times_{\cX}\ldots\times_{\cX}\cY$. Notice that the symmetric group $\fS_n$ acts on $\cY^{(n)}$ by permuting the components, leaving the projection $f^{(n)}:\cY^{(n)}\rightarrow\cX$ invariant.
	\begin{lemma}\label{lem: equidimensionality}
		The stack $\cY^{(n)}$ is equidimensional, and $\fS_n$ acts transitively on the set of its irreducible components.
	\end{lemma}
	\begin{proof}
		Notice that $\fS_n$ acts transitively on the fiber of $f^{(n)}$, hence it is enough to show that every irreducible component $\cZ$ of $\cY^{(n)}$ surjects onto $\cX$. However, this is clear as $f$ is both open and closed. See also~\cite[Section 1.1]{SY21}.
	\end{proof}
	\begin{remark}
		In particular, Lemma~\ref{lem: equidimensionality} shows that the stabilizer in $\fS_n$ of an irreducible component $\cZ$ of $\cY^{(n)}$ is independent of the component up to conjugation.
	\end{remark}
	We are ready to give a more geometric interpretation of the monodromy group of a covering map.
	\begin{proposition}\label{prop: main equivalence monodromy}
		Let $\cZ$ be an irreducible component of $\cY^{(n)}$, and let $\G$ be the subgroup of $\fS_n$ preserving $\cZ$. Then, $f^{(n)}:\cY^{(n)}\rightarrow\cX$ is a $\fS_n$-torsor, and the induced morphism $f_{\cZ}^{(n)}:\cZ\rightarrow\cX$ is a Galois closure of $f$, with Galois group $\G$. In particular, $\G=\Mon_f$.
	\end{proposition}
	\begin{proof}
		To verify that $f^{(n)}$ is a $\fS_n$-torsor we can work fppf-locally on the base, hence we can reduce to the case where $\cX$ is a scheme, where it is clear. In particular,
		\[
		\cX\simeq[\cY^{(n)}/\fS_n]\simeq[\cZ/\G],
		\]
		thus $f_{\cZ}^{(n)}$ is a Galois cover with Galois group $\G$. Since $f_{\cZ}^{(n)}$ factors through $\cY\rightarrow\cX$, we get an induced map
		\[
		\begin{tikzcd}
			g:\cZ\arrow[r] & \widetilde{\cY}
		\end{tikzcd}
		\]
		with $\widetilde{f}\circ g=f_{\cZ}^{(n)}$, where $\widetilde{f}:\widetilde{\cY}\rightarrow\cX$ is the Galois closure of $f$. To conclude, it is enough to show that $g$ is an isomorphism. Since $g$ is a cover and $\cZ$ is irreducible, it is enough to find a section to it. The cover $h:\widetilde{\cY}\rightarrow\cY$ induces a surjection $h^n:\widetilde{\cY}\times_{\cX}\ldots\times_{\cX}\widetilde{\cY}\rightarrow\cY\times_{\cX}\ldots\times_{\cX}\cY$. Let $\cU:=(h^n)^{-1}(\cY^{(n)})$, and denote by $h^{(n)}$ the restriction of $h^n$ to $\cU$.
		Since $\cX\simeq[{\widetilde{\cY}}/\Mon_f]$, we have that
		\[
		\widetilde{\cY}\times_{\cX}\ldots\times_{\cX}\widetilde{\cY}\simeq\widetilde{\cY}\times(\Mon f)^{n-1},
		\]
		where $(\widetilde{y},g_1,\ldots,g_{n-1})\in\widetilde{\cY}\times(\Mon f)^{n-1}$ is sent to $(\tilde y, g_1 \cdot \tilde y,\ldots, g_{n-1}\cdot \tilde y)$.
		Therefore, $\cU\simeq\widetilde{\cY}\times W$,
		where $W$ is a (discrete) subset of $(\Mon_f)^{n-1}$. In particular, every irreducible component of $\cU$ is isomorphic to $\widetilde{\cY}$ under any projection. Taking a copy of $\widetilde{\cY}$ in $\widetilde{\cY}^{(n)}$ that maps to $\cZ$ under $h$, this yields the desired section.
	\end{proof}
	\begin{example}\label{exm: interesting combinatorial example}
		The simplest non-schematic example to which we can apply the previous proposition is the $\G$-Galois cover $f:\cY=\Spec k\rightarrow B\G=\cX$, where $\G$ is a finite étale group scheme of order $n$ over an algebraically closed field $k$. It is possible to show that $\cY^{(n)}$ is
		isomorphic to the quotient $G^{(n)}/G$ of the subset $\G^{(n)}$ of $\G^{n}$ consisting of $n$-tuples of distinct elements, by the diagonal action of $G$. The action of $\fS_n$ is by permutation of the order of the elements. It follows that $\cY^{(n)}$ is simply the union of $(n-1)!$ points, and Proposition~\ref{prop: main equivalence monodromy} implies that the stabilizer in $\fS_n$ of a point in $\cY^{(n)}$ is isomorphic to $\G$. This can be seen directly from the description above. Indeed, a choice of a point in $\G^{(n)}/\G$ and a lift of it $(g_1,\ldots,g_n)\in\G^{(n)}$ induces an injection $\G\hookrightarrow\fS_n$, where $g\mapsto\sigma_g$ such that $gg_i=g_{\sigma_g(i)}$. By construction, the image $\G\subset\fS_n$ is equal to the stabilizer of $[g_1,\ldots,g_n]\in\G^{(n)}/\G$ under the $\fS_n$-action.
	\end{example}
	\begin{corollary}\label{cor: independence open substack}
		Let $f:\cY\rightarrow\cX$ be a representable, finite, étale cover between irreducible algebraic stacks. Let $\cU\subset\cX$ be a non-empty open substack, and let $f_{\cU}:\cY_{\cU}\rightarrow\cU$ be the restriction of $f$ to $\cU$. Then, there is a natural isomorphism
		\[
		\Mon_f=\Mon_{f_{\cU}}.
		\]
	\end{corollary}
	\begin{proof}
		Applying Proposition~\ref{prop: main equivalence monodromy}, we can identify $\Mon_{f_{\cU}}$ with the subgroup of $\fS_n$ preserving a connected component $\cZ_{\cU}$ of $\cY_{\cU}^{(n)}$. Taking $\cZ$ to be the closure of $\cZ_{\cU}$ inside $\cY^{(n)}$, which is an irreducible component of $\cY^{(n)}$, and applying Proposition~\ref{prop: main equivalence monodromy} again, we get the desired result.
	\end{proof}
	\begin{remark}
		Notice that it is important that $\cY$ is irreducible, otherwise it could happen that $\cY_{\cU}$ is disconnected, causing the monodromy group of $f_{\cU}$ to be smaller. This cannot happen if $\cX$ is normal, as in that case $\cY$ is always irreducible.
	\end{remark}
	Corollary~\ref{cor: independence open substack} allows to extend the definition of monodromy group to branched covers between irreducible stacks.
	\begin{definition}\label{def: extension definition monodromy to branched coverings}
		Let $f:\cY\rightarrow\cX$ be a presentable, generically finite and generically étale cover between irreducible stacks, and let $\cU\subset\cX$ be any non-empty open substack over which the restriction $f_{\cU}$ of $f$ is finite and étale. We define the monodromy group of $f$ as $\Mon_{f_{\cU}}$.
	\end{definition}
	The previous corollary shows that this definition is independent of the open substack chosen. Proposition~\ref{prop: main equivalence monodromy} has also another consequence.
	\begin{corollary}\label{cor: invariance monodromy pullback flat maps with connected fibers}
		Let $\pi:\cW\rightarrow\cX$ be a flat morphism with geometrically irreducible fibers. Let $f:\cY\rightarrow\cX$ be a representable, generically finite, generically étale cover between irreducible algebraic stacks, and let $f_{\cW}:\cV\rightarrow\cW$ be the pullback of $f$ along $\pi$. Then,
		\[
		\Mon_f=\Mon_{f_{\cW}}.
		\]
	\end{corollary}
	\begin{proof}
		Notice that the hypothesis imply that $\cV$ and $\cW$ are irreducible as well. Then, by last sentence in Definition~\ref{def: extension definition monodromy to branched coverings}, we can assume $f:\cY\rightarrow\cX$ to be finite and étale. Now, let $\cZ$ be an irreducible component of $\cY^{(n)}$, and $\cZ_{\cW}\subset\cV^{(n)}$ its preimage along $\pi$. Then, $\cZ_{\cW}$ is an irreducible component of $\cV^{(n)}$, and Proposition~\ref{prop: main equivalence monodromy} shows that the two monodromy groups are equal.
	\end{proof}
	We conclude the subsection by showing how we can generically factor any cover into a FPR cover and a homeomorphism. For simplicity, we prove the result for Deligne-Mumford stacks separated and of finite type over a field. See~\cite[Section 3]{EHKV01} or~\cite{Gir65} for the notion of a banded gerbe.
	\begin{lemma}\label{lem: factorization FPR and gerbe}
		Let $\cY\rightarrow\cX$ be a representable, finite, étale morphism of Deligne-Mumford stacks separated and of finite type over an algebraically closed field $k$, with $\cX$ irreducible. Suppose that the moduli map $\pi_{\cX}:\cX\rightarrow X$ is fppf. Then, there exists a commutative diagram with cartesian square
		\begin{equation}\label{eq: diag factorization}
			\begin{tikzcd}
				\cY\arrow[d,"f_1"]\arrow[dd,bend right=40,"f"']\arrow[rd,"\pi_{\cY}"]\\
				\cY'\arrow[d,"f_2"]\arrow[r,"\pi_{\cY'}"] & Y\arrow[d,"g"]\\
				\cX\arrow[r,"\pi_{\cX}"] & X
			\end{tikzcd}
		\end{equation}
		where $\pi_{\cY}$ denotes the corresponding fppf moduli map, $f_1$ is a finite étale homeomorphism, and $f_2$ is FPR.
		Moreover, if $\pi_{\cX}:\cX\rightarrow X$ is banded by a group scheme $\G$ over $k$, and $x$ is a geometric point of $\cX$ with residual gerbe $\iota:B\G\hookrightarrow\cX$, then there exists an exact sequence
		\begin{equation}\label{eq: exact sequence}
			\begin{tikzcd}
				\G\arrow[r,"\omega_{x}"] & \Mon_{f}\arrow[r] & \Mon_{f_2}\arrow[r] & 0.
			\end{tikzcd}
		\end{equation}
	\end{lemma}
	\begin{proof}
		By Proposition~\ref{prop: monotonous gerbe case}, we know that $\pi_{\cY}:\cY\rightarrow Y$ which is an fppf gerbe; as both stacks are Deligne-Mumford stacks, the moduli maps are actually étale. Moreover, both $X$ and $Y$ are separated, as $\cX$ and $\cY$ are. Let $g:Y\rightarrow X$ be the induced morphism between moduli spaces, and define $\cY'$ as in diagram~\eqref{eq: diag factorization}. Then, $g$ is étale because both $\pi_{\cY}$ and $g\circ\pi_{\cY}$ are, with $\pi_{\cY}$ being surjective as well. Clearly, $g$ is quasi-finite, and it is also proper. Indeed, it is separated as both $Y$ and $X$ are, it is clearly of finite type, and it is universally closed as $\pi_{\cY}$ is surjective and the composite $\cY\rightarrow X$ is proper. It follows that $g$ is a covering map and $f_2$ is an FPR covering map. Since $\pi_{\cY}$ and $\pi_{\cY'}$ are coarse moduli space maps, hence homeomorphisms, also $f_1$ is a homeomorphism. Now, we prove the last part of the statement. The surjection follows from the fact that $f$ factors through $f_2$. Recall that $\Mon_{f_2}=\Mon_{g}$ by Proposition~\ref{prop: correspondence FPR}, hence we can replace the term $\Mon_{f_2}$ with $\Mon_g$ in~\eqref{eq: exact sequence}. Then, by~\cite[Theorem 7.11]{Noo04} there is a commutative diagram
		\[
		\begin{tikzcd}
			& \pi_1(\cY)\arrow[d,"f_*"]\arrow[r] & \pi_1(Y)\arrow[d,"g_*"]\arrow[r] & 0\\
			\G\arrow[r] & \pi_1(\cX)\arrow[r,"\phi"] & \pi_1(X)\arrow[r] & 0
		\end{tikzcd}
		\]
		where we are suppressing the base points from the notation. With the notation of Definition~\ref{def: monodromy for stacks}, $\phi^{-1}(N_g)\subset\mathrm{image}(f_*)$ is normal in $\pi_1(\cX)$, hence it is contained in $N_f$. This shows that $N_f$ surjects onto $N_g$, hence the natural homomorphism $G\rightarrow\ker(\Mon_{f}\rightarrow\Mon_g)$ is surjective, by diagram chase. This concludes.
	\end{proof}
	\begin{remark}\label{rmk: generically always monotonous gerbe}
		By Remark~\ref{rmk: generic gerbe}, for every covering map $f:\cY\rightarrow\cX$ of irreducible, separated algebraic stacks of finite type over an algebraic closed field $k$ there is always a non-empty open substack $\cU\subset\cX$ over which $f_{\cU}$ admits a factorization as above. Then, using Corollary~\ref{cor: independence open substack}, one can always reduce to a situation as above.
	\end{remark}
	\subsection{First Examples}\label{subsec: first examples}
	In this subsection we present some examples to motivate the use of stacks to compute monodromy. We will see other examples in Section~\ref{sec: G cubic surfaces}, which will not have a clear counterpart with coarse moduli spaces.
	\subsubsection{Lines on Smooth Cubic Surfaces}\label{subsec: lines on cubics generic case}
	For simplicity, we work over $\C$, but the result holds for instance over any algebraically closed field of characteristic 0 or greater than $5$.
	
	The first example we consider is the classical problem of the 27 lines on a smooth cubic surface. It was shown by Cayley and Salmon (\cite{Cay49}, \cite{Sal49}) that every smooth cubic surface in $\P^3$ contains exactly 27 lines. If we denote by $\M$ the coarse moduli space of smooth cubic surfaces and by $\L$ the variety parametrizing lines on cubic surfaces, we get a finite, generically étale cover $\L\rightarrow\M$ of degree 27, and we can ask what is the monodromy in this case. This question was studied by Jordan in 1870 (\cite{Jor70}) and a century later by Harris (\cite{Har79}). Here, we present an alternative way to compute the monodromy groups, using the techniques of this section.
	\begin{definition}\label{def: stack smooth cubic surfaces}
		A smooth cubic surface is a smooth Del Pezzo surface of degree 3. A morphism $f:S\rightarrow B$ of schemes is a family of smooth cubic surfaces if it is fppf, proper, and with smooth cubic surfaces as fibers. We denote by $\cM$ the stack of smooth cubic surfaces and by $\pi_{\cM}:\cM\rightarrow\M$ its coarse moduli space.
	\end{definition}
	Recall that given a family of smooth cubic surfaces $f:S\rightarrow B$, the relative dualizing sheaf $\omega_{S/B}$ is locally free of rank 1, and its dual $\omega_{S/B}^{\vee}$ is relatively ample. The pushforward $f_*\omega_{S/B}^{\vee}$ is a locally free sheaf of rank 4, and the canonical homomorphism $f^*f_*\omega_S^{\vee}\rightarrow\omega_S^{\vee}$ is surjective. This induces a closed immersion $S\hookrightarrow\P(f_*\omega_S^{\vee})$ over $B$, realizing $S$ as a relative cubic surface in $\P^3$. In particular, one can talk about families of lines on $S$.
	\begin{definition}\label{def: stack smooth cubic surfaces with a line}
		We denote by $\cL$ the stack parametrizing cubic surfaces together with a line, where an arrow $(L_1\subset S_1\xrightarrow{f_1} B_1)\rightarrow(L_2\subset S_2\xrightarrow{f_2}B_2)$ is a cartesian diagram
		\[
		\begin{tikzcd}
			L_1\arrow[d,hookrightarrow]\arrow[r] & L_2\arrow[d,hookrightarrow]\\
			S_1\arrow[r,"\phi'"]\arrow[d,"f_1"'] & S_2\arrow[d,"f_2"]\\
			B_1\arrow[r,"\phi"'] & B_2.
		\end{tikzcd}
		\]
		Here, $f_i$ are families of smooth cubic surfaces, each with a family of lines $L_i$.
		Let $\pi_{\cL}:\cL\rightarrow\L$ be the coarse moduli space of $\cL$. We denote by $\cF:\cL\rightarrow\cM$ and $\F:\L\rightarrow\M$ the morphisms that forget the datum of the line on the surface.
	\end{definition}
	\begin{remark}
		It is well known that $\cM$ and $\cL$ are smooth Deligne-Mumford stack of dimension 4, and $\cM$ is irreducible. The morphism $\cF:\cL\rightarrow\cM$ is a representable, finite, étale cover of degree 27. It is also known that $\cL$ is irreducible, but this is equivalent to showing that the monodromy group acts transitively on the fiber.
	\end{remark}
	A generic smooth cubic surface has trivial automorphism group, hence $\cM\rightarrow\M$ and $\cL\rightarrow\L$ are generically isomorphisms. In particular, the two covers $\cF$ and $\F$ have the same monodromy group by Corollary~\ref{cor: independence open substack}. Notice that $\L\rightarrow\M$ is only generically étale, and branched over the locus of $\M$ corresponding to surfaces that admit non-trivial automorphisms. We are then interested in computing $\Mon_{\cF}$.
	
	It is well known that the monodromy group respects the intersection product. This means that, if we realize $\Mon_{\cF}$ as a subgroup of $\fS_{27}$ as in Proposition~\ref{prop: main equivalence monodromy}, then $\Mon_{\cF}$ is contained in the subgroup that preserves the intersection between the lines. In~\cite[Section 3.3]{Har79}, Harris proved that this subgroup is isomorphic to the Weyl group $\W(E_6)$, of order $51840$. We are left with showing that also the other containment holds; see~\cite[Section 3.3]{Har79} for another proof. Our idea is to leverage the fact that we are working with stacks to produce monodromy elements by looking at the automorphism groups of cubic surfaces, using the morphism $\omega_x$ from the hidden fundamental group to the usual one.
	
	\begin{proposition}\label{prop: monodromy general case}
		With the notation above, we have
		\begin{equation}\label{eq: monodromy general case}
			\Mon_{\F}=\Mon_{\cF}\simeq\W(E_6).
		\end{equation}
	\end{proposition}
	\begin{proof}
		We generate monodromy by looking at the stabilizers; see~\cite[Table 9.6]{Dol12} and~\cite[Table 5]{BLP23} for the automorphism groups of smooth cubic surfaces. The key observation is the following. First, for every smooth cubic surface $S$ corresponding to a geometric point $x$ of $\cM$, its automorphism group maps non-canonically to $\Mon_{\cF}$ via the morphism $\omega_x:\Aut(x)\rightarrow\pi_1(\cX,x)$ defined in~\ref{sec: fundamental group}; a rigorous construction that takes the dependence on $x$ into account is presented below. Second, $\Aut(S)$ acts faithfully on the lines on $S$, thus the homomorphism $\Aut(S)\rightarrow\Mon(\cF)$ is injective. To make this precise, fix a geometric point $x$ of $\cM$, and for every group in~\cite[Table 9.6]{Dol12} choose a geometric point representing a smooth cubic surface with that automorphism group. For each such point $y$, choose a path from $x$ to it, which induces an isomorphism $\pi_1(\cY,y)\xrightarrow{\simeq}\pi_1(\cX,x)$. The composite of this isomorphism with $\omega_y$ yields injections from the stabilizer group into the monodromy group; we denote by $\G$ the subgroup of $\Mon_{\cF}\subset\W(E_6)$ generated by the images of all these maps. Since two automorphism groups have order 648 and 120 (corresponding to the Fermat and Clebsch cubic respectively), $\G$ is a subgroup of $\W(E_6)$ of order divisible by their least common multiple $3240=2^4\cdot3^4\cdot 5$, hence of index dividing 8. Thanks to the classification of maximal subgroups of $\W(E_6)$ (\cite[Theorem 9.5.2]{Dol12}), it follows that either $\G=\W(E_6)$ or $\G$ is contained in the simple normal index 2 subgroup $U_4(2)$. Notice that there is a surface whose automorphism group is isomorphic to the cyclic group $\mC_8$ of order 8.
		As $U_4(2)$ does not have elements of order 8, it follows that $\G=\W(E_6)$.
	\end{proof}
	\begin{remark}
		Although the proof above is relatively short, we are using some aspects of the description of the automorphism groups of smooth cubic surfaces and the knowledge of the maximal subgroups of $\W(E_6)$. For the first problem, notice that it is actually enough to get a lower bound on the automorphism groups of some cubic surfaces, producing highly symmetric examples, which we can do explicitly. For the second problem, one could choose paths from a fixed point $x$ to points representing the chosen highly symmetric surfaces to explicitly compute the action of the automorphism groups on the fiber of $\cF:\cL\rightarrow\cM$ over $x$, and check that this yields the whole $\W(E_6)$.
	\end{remark}
	Notice that in this example the stack $\cM$ that we are working with is generically a scheme, which slightly simplifies the computation. In Section~\ref{sec: monodromy G cubic surfaces} and in the next subsection, we will see examples where the general stabilizer is non-trivial; in Section~\ref{sec: G cubic surfaces} we will develop other tools to deal with these cases.
	\subsubsection{Flexes of Plane Cubics}
	Consider the enumerative problem of counting how many flex points are there on a complex smooth plane curve of degree 3. It is well known that there are nine flex points, and we can compute the monodromy of the associated cover. This was done in~\cite[Subsection 2.2]{Har79}.
	
	Our approach is different. Denote by $\cX$ the moduli stack of smooth (embedded) plane cubics up to projective equivalence, that is,
	\[
	\cX=[(\P^9\setminus\Delta)/\PGL_3],
	\]
	where $\Delta\subset\P^9$ is the locus corresponding to singular cubic forms. Notice that $\cX$ is smooth and irreducible, and it admits a moduli space induced by the $j$-invariant. Let $\cY$ be the moduli stack of pairs $(C\subset\P^2,p)$ where $p\in C\subset\P^2$ is a flex, and arrows are projective transformations preserving the marked point. Notice that $\cY$ is isomorphic to the stack of elliptic curves $\cM_{1,1}$, in particular it is a smooth Deligne-Mumford stack, and the morphism $f:\cY\rightarrow\cX$ that forgets the flex point is a representable, finite, étale cover of degree 9. We recall some basic facts about automorphisms of plane cubics, flexes and the Hesse form.
	
	\begin{lemma}\label{lem: automorphism plane cubics}
		Let $C$ be a smooth plane cubic in $\P^2$ over $\C$. Then, $C$ is projectively equivalent to a plane cubic in Hesse form, that is, defined by a polynomial
		\[
		g(x,y,z)=x^3+y^3+z^3-3k\cdot xyz
		\]
		for some parameter $k\in\C$, $k^3\not=1$. In this form, the 9 flexes are $(0:1:-\gamma)$, $(-\gamma:0:1)$, and $(1:-\gamma:0)$, with $\gamma^3=1$. Every projective automorphism of $C$ preserves the set of flexes, and acts faithfully on them. Thus, every geometric point $y$ of $\cY$ induces a split exact sequence
		\begin{equation}\label{eq: exact sequence automorphisms plane cubics}
			\begin{tikzcd}
				0\arrow[r] & \mC_3\times\mC_3\arrow[r] & \Aut_{\cX}(f(y))\arrow[r] & \Aut_{\cY}(y)\arrow[r] & 0,
			\end{tikzcd}
		\end{equation}
		where $\mC_3\times\mC_3$ acts simply transitive on the set of flexes of the plane cubic corresponding to $f(y)$. Finally, $\cX$ has generic stabilizer of order 18.
	\end{lemma}
	\begin{proof}
		See~\cite[Lemma 2.4, Theorem 2.12 and Corollaries 2.14, 3.10]{BM16}. The group $\mC_3\times\mC_3$ comes from the fact that the 9 flexes form a group isomorphic to it, using the abelian group structure of the elliptic curve. The splitting of the exact sequence~\eqref{eq: exact sequence automorphisms plane cubics} is given by the natural inclusion
		\[
		\begin{tikzcd}
			\Aut_{\cY}(y)\arrow[r,hookrightarrow] & \Aut_{\cX}(f(y)).
		\end{tikzcd}
		\]
		The last assertion follows from the exact sequence and the fact that the elliptic involution of $(C\subset\P^2,p)$ preserves the flexes and fixes $p$.
	\end{proof}
	\begin{remark}
		By Lemma~\ref{lem: factorization FPR and gerbe}, we know that generically we can split $f:\cY\rightarrow\cX$ as the composite of a homeomorphism and an FPR morphism as in~\eqref{eq: diag factorization}. For a general $x=f(y)$, the exact sequence~\eqref{eq: exact sequence automorphisms plane cubics} reads
		\[
		\begin{tikzcd}
			0\arrow[r] & \mC_3\times\mC_3\arrow[r] & \Aut_{\cX}(x)\arrow[r] & \mC_2\arrow[r] & 0,
		\end{tikzcd}
		\]
		with a section $\mC_2\simeq\Aut_{\cY}(y)\rightarrow\Aut_{\cX}(x)$. As the first group in the exact sequence acts simply transitively, the factorization of $f$ in Lemma~\ref{lem: factorization FPR and gerbe} is the trivial one, meaning that $\cY'=\cX$ and that the induced cover $g$ of $X$ is the identity. Indeed, both $\cY\simeq\cM_{1,1}$ and $\cX$ have the same moduli space $\A^1$ via the $j$-invariant. This is in contrast with what we had in the previous example.
	\end{remark}
	Before computing the monodromy of $f$, we recall a lemma that forces the monodromy group to be rather small.
	\begin{lemma}\label{lem: constraint monodromy plane cubics}
		The monodromy group is a subgroup of the affine special linear group $\ASL_2(\mathbb{F}_3)$, which has order 216.
	\end{lemma}
	\begin{proof}
		See~\cite[Section 1.2]{Har79}. For a sketch of an alternative proof, see Remark~\ref{rmk: simplification background plane cubics}.
	\end{proof}
	\begin{proposition}\label{prop: monodromy plane cubics}
		The monodromy group of $f:\cY\rightarrow\cX$ is $\Mon_{f}=\ASL_2(\mathbb{F}_3)$.
	\end{proposition}
	\begin{proof}
		Since the automorphism group of a plane cubic acts faithfully on the set of flexes, $\Aut_{\cX}(x)$ injects in the monodromy group under $\omega_x$, by the same argument as for Proposition~\ref{prop: monodromy general case}. Thanks to Lemma~\ref{lem: automorphism plane cubics}, we know that the cardinality of $\Aut_{\cX}(x)$ is $|\Aut_{\cX}(x)|=9\cdot|\Aut_{\cY}(y)|$ for any $y$ in $\cY$ mapping to $x$. As $\Aut_{\cY}(y)$ is just the group of automorphism of an elliptic curve, we get that it is cyclic of order 2, 4 or 6, depending on the $j$-invariant. By choosing paths from a fixed point $x$ of $\cX$ to a finite set of points that have large automorphism group, as in the proof of Proposition~\ref{prop: monodromy general case}, we obtain a subgroup $\G$ of $\Mon_f\subset\ASL_2(\mathbb{F}_3)$ of order at least $9\cdot12=158$, hence of index at most 2 in $\ASL_2(\mathbb{F}_3)$, and containing the kernel of $\ASL_2(\mathbb{F}_3)\rightarrow\SL_2(\mathbb{F}_3)$. Since $\SL_2(\mathbb{F}_3)$ does not have index 2 subgroups, we get that $\G=\ASL_2(\mathbb{F}_3)$, and we are done.
	\end{proof}
	\begin{remark}\label{rmk: simplification background plane cubics}
		As for Proposition~\ref{prop: monodromy general case}, we do not need all the results cited above to prove Proposition~\ref{prop: monodromy plane cubics}. Indeed, instead of using Lemma~\ref{lem: automorphism plane cubics}, it is enough to compute the automorphism group of some highly symmetric plane cubic. Regarding Lemma~\ref{lem: constraint monodromy plane cubics}, one could prove a weaker version stating that the monodromy is contained in $\AGL_2(\mathbb{F}_3)$. Indeed, the stabilizer of a flex point in the mondromy group of $f$ preserves the group structure of the set of flexes with one fixed point, which is $\mC_3\times\mC_3$, hence it is contained in $\GL_2(\mathbb{F}_3)$. Since the subgroup $\mC_3\times\mC_3$ appearing in the exact sequence~\eqref{eq: exact sequence automorphisms plane cubics} acts transitively on the set of flexes and it injects into the monodromy, we deduce that the monodromy group injects into $\AGL_2(\mathbb{F}_3)$. To show that $\Mon_f\subset\ASL_2(\mathbb{F}_3)$, one can consider the same problem over $\R$ instead of $\C$. Then, complex conjugation gives an $\R$-automorphism that contributes to the monodromy of the cover, over $\R$. By degree reasons, extending the field to $\C$ annihilates only that monodromy element, hence causing the monodromy to be contained in a proper subgroup of $\AGL_2(\mathbb{F}_3)$. From this and the information we have about the group $\G$ appearing in the above proof, one can deduce the desired result. See~\cite[pages 696, 697]{Har79} for a similar discussion over $\R$.
	\end{remark}
	To avoid working with moduli stacks, one could have considered the generically finite and étale cover $I\rightarrow\P^9$, where $I$ is the incident variety of flex points of plane cubics; see~\cite[Subsection 2.2]{Har79}. The restriction of this cover to the locus of smooth plane cubics is simply the pullback of $f:\cY\rightarrow\cX$ along the quotient map $\P^9\setminus\Delta\rightarrow\cX$. Since $\PGL_3$ is irreducible, by Corollary~\ref{cor: invariance monodromy pullback flat maps with connected fibers} the monodromy group is the same.
	\section{Stack of Objects with $\G$-Actions}\label{sec: G cubic surfaces}
	Given a representable, finite, étale cover $f:\cY\rightarrow\cX$ between irreducible algebraic stacks, we may ask what is the monodromy of the restriction of $f$ to the locus $\cX^{\G}\subset\cX$ of points whose automorphism group contains a prescribed subgroup $\G$. This is particularly interesting when the cover arises as a solution to an enumerative problem. This question was raised in~\cite{BR24}, where the authors solved one instance of this problem; see also Subsection~\ref{subsec: S4 cubic surfaces} for a different treatment of the same example.
	
	Let us first informally motivate the content of this section; see Subsection~\ref{subsec: def main stacks} for the rigorous definitions. Suppose that $\cX$ has generically trivial stabilizer, hence $f:\cY\rightarrow\cX$ is generically FPR, see Definition~\ref{def: FPR covers}, and suppose further that the two stacks admit moduli spaces $\pi_{\cX}:\cX\rightarrow X$ and $\pi_{\cY}:\cY\rightarrow Y$. Then, the induced map $g:Y\rightarrow X$ is a generically finite and generically étale cover of the same degree. However, restricting the cover to the the image $X^{\G}$ of $\cX^{\G}$ in $X$, one could get a cover which is generically finite, generically étale of a smaller degree, hence not necessarily reflecting what one would want to parametrize. This is exemplified in Example~\ref{exm: bad cover cubic surfaces with automorphisms} below. In particular, we would loose the part of the monodromy group coming from the generic stabilizer. This is one reason why it generally helps to directly consider the cover at the level of moduli stacks.
	
	As we have already seen in Section~\ref{sec: fundamental group}, the presence of non-trivial stabilizers is also very useful to generate elements of the monodromy groups.
	
	Another problem is that $X^{\G}$ tends to be non-normal; in particular, a cover $Y^{\G}\rightarrow X^{\G}$ can be connected but not irreducible. Therefore, we cannot apply Corollary~\ref{cor: independence open substack}, which also introduces an ambiguity in how to define the monodromy group of $Y^G\rightarrow X^G$ if it is ramified. The same issue arises for $\cX^{\G}$.
	\begin{example}\label{exm: bad cover cubic surfaces with automorphisms}
		An example of this phenomenon is the cover $\cF:\cL\rightarrow\cM$ of the moduli stack $\cM$ of smooth cubic surfaces induced by the 27 lines, see Subsection~\ref{subsec: lines on cubics generic case}. The generic cubic surface that carries a non-trivial automorphism has automorphism group equal to $\mC_2$, and has exactly one \emph{Eckardt point}. An Eckardt point is a point where three of the 27 lines meet, see Definition~\ref{def: Eckardt points}. Over the locus $\cZ$ of points with exactly 1 Eckardt point and automorphism group $\mC_2$, the cover $\cL\rightarrow\cM$ is reducible, and it breaks into a degree 3 cover parametrizing lines passing through the Eckard point, and a degree 24 cover parametrizing the remaining lines. The group $G:=\mC_2$ acts trivially on the first set of lines, hence the corresponding cover is FPR. On the other hand, $G$ acts freely on the second set of lines, hence the induced cover at the level of coarse moduli spaces is of degree 12. In other words, the corresponding factorization of $\cF|_{\cM^{\G}}$ coming from Lemma~\ref{lem: factorization FPR and gerbe} consists in a degree 2 Galois cover and a FPR cover of degree 12. Finally, notice that $\cL^{\G}$ is connected, even though it is reducible, as the stabilizer group of the point in $\cM^{\G}$ corresponding to the Fermat cubic acts transitively on the 27 lines. This phenomenon  arises as $\cM^{\G}$ is non-normal, as a cubic surface with exactly 2 Eckardt points can be deformed into a cubic surface with exactly 1 Eckardt point in two independent ways.
	\end{example}
	As explained above, $\cX^{\G}$ has some characteristics that make it undesirable to work with. In Subsection~\ref{subsec: def main stacks} we introduce a smooth moduli stack $\cX_{\G}$ which serves as a well-behaved replacement for $\cX^{\G}$. The two stacks are related by a finite morphism $\cX_{\G}\rightarrow\cX^{\G}$. We then study the monodromy group of the induced cover $\cY_{\G}\rightarrow\cX_{\G}$, and relate it to the monodromy groups of the covers $\cY^{\G}\rightarrow\cX^{\G}$ and $Y^{\G}\rightarrow X^{\G}$.
	In some cases, this gives also a way to find a concrete presentation of $\cX^{\G}$ as a quotient stack.
	\subsection{Definitions of the Main Stacks}\label{subsec: def main stacks}
	
	In this subsection we rigorously define the stacks $\cX^{\G}$ and $\cX_{\G}$. Since these notions are not well behaved in general, we will add some assumptions to simplify the treatment; however, not all are strictly necessary.
	\begin{setup}\label{set: setup fixed point stack}
		Throughout this section, we will always assume $\cX$ to be a separated Deligne-Mumford stack of finite type over an algebraically closed field $k$. In particular, $\cX$ admits a coarse moduli space $\pi_{\cX}:\cX\rightarrow X$, by Keel-Mori Theorem~(\cite[Corollary 1.3]{KM97} and~\cite[Theorem 1.1]{Con05}), and $\pi_{\cX}$ is a proper universal homeomorphism.
		We will always assume $\cX$ to be tame, that is,
		for every geometric point $x$ of $\cX$ the stabilizer group of $x$ is linearly reductive. We will denote by $\G$ a finite group of order coprime to the characteristic of $k$ (if positive), hence linearly reductive.
	\end{setup}
	\begin{remark}\label{rmk: base change of tame moduli space}
		Under these hypotheses, for every morphism of schemes $B\rightarrow X$, the fiber product $\cX\times_XB$ is a tame Deligne-Mumford stack with moduli space $B$, see~\cite[Corollaries 3.3, 3.4]{AOV08}.
	\end{remark}
	Defining a substack $\cX^{\G}\subset\cX$ that parametrizes objects admitting a $\G$-action is a delicate problem, and it is not clear what its moduli description should be. Therefore, we start by defining the stack $\cX_{\G}$, which comes with a finite forgetful morphism $\cX_{\G}\rightarrow\cX$, whose image will be $\cX^{\G}$.
	
	\begin{definition}\label{def: stack of objects with G action}
		A \emph{$\G$-object} of $\cX$ over a scheme $T$ is a pair $(x,\rho)$ where $x\in\cX(T)$ and $\rho:\G_T\hookrightarrow\uAut_{\cX}(x)$ is a closed immersion and a homomorphism of group schemes over $T$. We denote by $\cX_{\G}$ the category fibered in groupoids over $\Sch_{\text{fppf}}$ whose objects are $\G$-objects and an arrow $\alpha:(x_1,\rho_1)\rightarrow(x_2,\rho_2)$ is a $\G$-equivariant morphism $\alpha:x_1\rightarrow x_2$ in $\cX$. This means that for every $g\in\G$ the diagram
		\[
		\begin{tikzcd}
			x_1\arrow[d,"\rho_1(g)"']\arrow[r,"\alpha"] & x_2\arrow[d,"\rho_2(g)"]\\
			x_1\arrow[r,"\alpha"] & x_2
		\end{tikzcd}
		\]
		is commutative.
	\end{definition}
	\begin{lemma}\label{lem: closed immersion = universally injective}
		Let $T$ be a scheme, $x\in\cX(T)$ and $\rho:\G_T\rightarrow\uAut_{\cX}(x)$ be a morphism over $T$. Then, $\rho$ is a closed immersion if and only if it is universally injective.
	\end{lemma}
	\begin{proof}
		Recall that a closed immersion is always universally injective, and the converse holds if the morphism is further assumed to be proper and unramified, see~\cite[\href{https://stacks.math.columbia.edu/tag/04XV}{Lemma 04XV}]{Sta25}. On the other hand, $\G_T\rightarrow T$ is both proper and unramified, hence it is enough to show that the natural morphism $\varphi:\uAut_{\cX}(x)\rightarrow T$ satisfies the same properties. The statement follows from the fact that $\varphi$ is obtained as a base change from the diagonal morphism $\cX\rightarrow\cX\times\cX$, which is proper and unramified, as $\cX$ is a separated Deligne-Mumford stack.
	\end{proof} 
	The stack $\cX_{\G}$ has already been studied by Romagny in~\cite{Rom05} and~\cite{Rom22}, in greater generality. Specifically, he introduces the notion of fixed point stack associated to an action of a group $\G$ on a stack $\cX$. When the action is trivial, the fixed point stack parametrizes $\G$-objects of $\cX$, with the only difference that $\rho$ is not required to be a faithful action; he denotes this stack by $\cX[\G]$. Under our assumptions, $\cX[\G]$ is an algebraic stack with a representable map $\cX[\G]\rightarrow\cX$, and $\cX_{\G}$ is an open substack of it; Romagny denotes $\cX_{\G}$ by $\cX\{\G\}$.
	\begin{proposition}\label{prop: basic facts about X{G}}
		Let $\cX$ and $\G$ be as in Setup~\ref{set: setup fixed point stack}. Then $\cX_{\G}$ is a tame Deligne-Mumford stack, and the morphism
		\[
		\begin{tikzcd}
			\Phi_{\G}:\cX_{\G}\arrow[r] & \cX
		\end{tikzcd}
		\]
		that forgets the $\G$-action is representable by schemes, finite and unramified. Moreover, if $\cX$ is smooth, then also $\cX_{\G}$ is.
	\end{proposition}
	\begin{proof}
		See~\cite[Theorem 3.3, Corollary 3.11]{Rom05} and~\cite[Theorem 4.1.4, Theorem 4.3.6, Corollary 4.4.6]{Rom22}, where $\cX_{\G}=\cX\{\G\}$.
	\end{proof}
	We are ready to define the closed substack $\cX^{\G}\subseteq\cX$. We use the notion of scheme-theoretic image as defined in~\cite{EG15} and~\cite[\href{https://stacks.math.columbia.edu/tag/0CMH}{Tag 0CMH}]{Sta25}, which is a closed substack of the target.
	\begin{definition}\label{def: stack of G objects no action}
		We denote by $\cX^{\G}$ the closed substack of $\cX$ that is the image of $\Phi_{\G}:\cX_{\G}\rightarrow\cX$, and by $X^{\G}$ the image of $\cX^{\G}$ under the moduli space map $\pi_{\cX}$, which is a closed subscheme of $X$.
	\end{definition}
	Although the notation might be misleading, these stacks are not fixed point stacks with respect to some $G$-action, in the sense of Romagny (\cite{Rom05},~\cite{Rom22}).
	The following lemma provides a local description of $\cX^{\G}$.
	\begin{lemma}\label{lem: stack of G objects no action properties}
		The stacks $\cX_{\G}$ and $\cX^{\G}$ satisfy the following properties.
		\begin{enumerate}
			\item\label{item:1} The construction of $\cX_{\G}$ and $\cX^{\G}$ commutes with flat morphisms $\cY\rightarrow\cX$ that induce isomorphisms between automorphism groups.
			\item\label{item:2} Suppose $\cX=[U/\H]$ with $\H$ finite and linearly reductive group acting on a scheme $U$ separated and of finite type over $k$. For every injective homomorphism $\varphi:\G\hookrightarrow\H$, let $U^{\varphi}$ be the fixed point scheme of the induced action of $\G$ on $U$, and let $U^{\G}=\cup_{\varphi}U^{\varphi}$. Then, $U^{\varphi}$ and $U^{\G}$ are closed in $U$, and $\cX^{\G}=[U^{\G}/\H]$.
			\item\label{item:3} The morphism $\pi_{\cX^{\G}}:\cX^{\G}\rightarrow X^{\G}$ is a coarse moduli space map, and the induced morphism $\cX^{\G}\rightarrow\cX\times_{X}X^{\G}$ is a universal homeomorphism.
			\item\label{item:4} If $\cX$ is smooth then $\cX^{\G}$ is reduced.
		\end{enumerate}		
	\end{lemma}
	\begin{proof}
		Suppose given a flat morphism $f:\cY\rightarrow\cX$ such that for every scheme $T$ and object $y\in\cY(T)$ the map $\uAut_{\cY}(y)\rightarrow\uAut_{\cX}(f(y))$ is an isomorphism, and form the fibered product $\cZ:=\cX_{\G}\times_{\cX}\cY$; we need to prove that $\cZ\simeq\cY_{\G}$. Recall that an object of $\cZ(T)$ is a triple $((x,\rho),y,\alpha)$ where $(x,\rho)\in\cX_{\G}(T)$, $y\in\cY(T)$ and $\alpha:f(y)\rightarrow x$ is an isomorphism in $\cX$. Then,
		\[
		\begin{tikzcd}
			\G_T\arrow[r,hookrightarrow,"\rho"] & \uAut_{\cX}(x)\simeq\uAut_{\cX}(f(y))\simeq\uAut_{\cY}(y)
		\end{tikzcd}
		\]
		yields an element of $\cY_{G}(T)$. Therefore, $\cZ\rightarrow\cY$ factors through $\cY_{\G}$. The inverse $\cY_{\G}\rightarrow\cZ$ sends a $T$-object $(y,\rho)$ to the triple $((f(y),\rho),y,\mathrm{id})$, where $\rho$ is seen as a injection into $\uAut_{\cX}(x)$ thanks to the isomorphisms above. Since the formation of the scheme-theoretic image commutes with flat base change by~\cite[\href{https://stacks.math.columbia.edu/tag/0CMH}{Tag 0CMH}]{Sta25}, property~\eqref{item:1} holds for $\cX^{\G}$ as well.
		
		Now, we prove point~\eqref{item:2}. The fact that $U^{\varphi}$ is closed follows from~\cite[Proposition A.8.10]{CGP15}. Since the set of injective homomorphisms $\varphi:\G\hookrightarrow\H$ is finite, $U^{\G}$ is closed as well. Now, we show that the morphism $\Phi_{\G}:\cX_{\G}\rightarrow\cX$ factors through the closed substack $\cZ:=[U^{\G}/\H]\subset\cX$. Let $(x,\rho)\in\cX_{\G}(T)$, thus $x\in\cX(T)$ corresponds to an $\H$-torsor $p:P\rightarrow T$ and an $\H$-equivariant morphism $u:P\rightarrow U$. To show that $x\in\cZ(T)$, we can pass to fppf covers and take connected components. In particular, we can assume that $P\simeq T\times\H$ is the trivial torsor, and $T$ to be connected. Since 
		$\H$ is a finite group, we have $\Aut(x)\subset\H$ by sending an automorphism $\psi$ of $T\times\H$ to the element $h\in\H$ such that $\psi(T\times\{1\})=T\times\{h\}$. This yields an injective homomorphism $\varphi:\G\hookrightarrow\H$, and by construction $u:T\times\H\rightarrow U$ factors through $\cup_{h\in\H}U^{h\varphi h^{-1}}$. This shows that $\Phi_{\G}$ factors through $\cX_{\G}\rightarrow[U^{\G}/\H]$, which is surjective, hence $\cX^{\G}=[U^{\G}/\H]$.
		
		For point~\eqref{item:3}, recall that $\cX\times_{X}X^{\G}\rightarrow X^{\G}$ is a coarse moduli space map by Remark~\ref{rmk: base change of tame moduli space}. Let $\cX^{\G}\rightarrow Z$ be the coarse moduli space map of $\cX^{\G}$, and let $f:Z\rightarrow X^{\G}$ be the morphism induced by the universal property of $Z$. We claim that $f$ is a monomorphism. Suppose given two morphisms $g_1,g_2:W\rightarrow Z$ such that $f\circ g_1=f\circ g_2$, and construct the cartesian diagram
		\[
		\begin{tikzcd}
			\cW\arrow[r,shift right,"\widetilde{g}_2"']\arrow[r,shift left,"\widetilde{g}_1"]\arrow[d,"\pi_{\cW}"'] & \cX^{\G}\arrow[r,hookrightarrow]\arrow[d] & \cX\times_{X}X^{\G}\arrow[d]\\
			W\arrow[r,shift right,"g_2"']\arrow[r,shift left,"g_1"] & Z\arrow[r] & X^{\G}.
		\end{tikzcd}
		\]
		It follows that $\widetilde{g}_1=\widetilde{g}_2$; since $\pi_{\cW}:\cW\rightarrow W$ is a coarse moduli space by Remark~\ref{rmk: base change of tame moduli space}, this implies that $g_1=g_2$.	Then, $f$ is a proper monomorphism, hence a closed immersion by~\cite[\href{https://stacks.math.columbia.edu/tag/04XV}{Tag 04XV}]{Sta25}. Since by definition $X^{\G}$ is the image of $\cX^{\G}$ under $\pi_{\cX}$, we get $Z=X^{\G}$, as wanted. Finally, the closed immersion $\cX^{\G}\hookrightarrow\cX\times_{X}X^{\G}$ is a universal homeomorphism, for instance by~\cite[\href{https://stacks.math.columbia.edu/tag/0H2M}{Tag 0H2M}]{Sta25}.
		
		By Proposition~\ref{prop: basic facts about X{G}}, if $\cX$ is smooth then $\cX_{\G}$ is smooth as well, thus reduced. Therefore, the image $\cX^{\G}$ of $\Phi_{\G}$ is also reduced, proving property~\eqref{item:4}.
	\end{proof}
	Unfortunately, it is not clear how to characterize the objects parametrized by $\cX^{\G}$, contrarily to what happens for $\cX_{\G}$, showcasing another advantage of working with the latter.
	
	Other than associating stacks $\cX_{\G}$ to linearly reductive groups $\G$, we can associate morphisms $\cX_{\G}\rightarrow\cX_{\H}$ to injective homomorphisms $\H\hookrightarrow\G$.
	\begin{definition}\label{def: action of AutG}
		Let $\psi:\H\hookrightarrow\G$ be an injective homomorphism of finite linearly reductive group schemes over $k$. Then, this induces a map
		\[
		\begin{tikzcd}
			\Phi_{\psi}:\cX_{\G}\arrow[r] & \cX_{\H}
		\end{tikzcd}
		\]
		sending an object $(x,\rho:\G\hookrightarrow\Aut(f))\in\cX_{\G}$ to $(x,\rho\circ\psi:\H\hookrightarrow\Aut(f))$, while the arrows stay the same as they are also $\H$-equivariant by construction.
	\end{definition}
	\begin{remark}
		Notice that in the above hypothesis there are closed immersions $\cX^{\G}\hookrightarrow\cX^{\H}$ and $X^{\G}\hookrightarrow X^{\H}$. Indeed, $\cX^{\G}=(\cX^{\H})^{\G}$, and $\cX^{\H}$ satisfies the properties of Setup~\ref{set: setup fixed point stack}.
	\end{remark}
	\begin{remark}\label{rmk: special case Phi H=0 H=G}				
		When $\H=\{1\}$, we omit both $\psi$ and $H$ from the notation, and we recover the map $\Phi_{\G}:\cX_G\rightarrow\cX^{\G}\subset\cX$.
		
		In the case where $\H=\G$, we have $\psi\in\Aut(\G)$ and $\Phi_{\psi}$ is an automorphism with inverse $\Phi_{\psi^{-1}}$.
		This defines an action of $\Aut(\G)$ on $\cX_{\G}$.
		We can restrict the action of $\Aut(\G)$ to the group of inner automorphisms $\Inn(\G)$, in particular to $\G$ under the surjective map 
		\[
		\begin{tikzcd}[row sep=tiny]
			\G\arrow[r] & \Inn(\G), & g\arrow[r,mapsto] & \psi_g=(h\mapsto ghg^{-1})
		\end{tikzcd}
		\]
		with kernel the centralizer $\Zen(\G)$ of $\G$. Notice that the action of $\G$ on $\cX_{\G}$ is trivial. To see this, let $g\in\G$ and $(x,\rho)$ be an object in $\cX_{\G}(T)$. Then, $(x,\rho)$ is isomorphic to $(x,\rho\circ\psi_g)$ via $\rho(g):x\rightarrow x$, as
		\[
		(\rho\circ\psi_g)(h)\circ\rho(g)=\rho(g)\circ\rho(h)\circ\rho(g^{-1})\circ\rho(g)=\rho(g)\circ\rho(h)
		\]
		by construction, hence the map is equivariant. This defines a 2-isomorphism between the $\G$-action and the projection $G\times_k\cX_G\rightarrow\cX_G$; this is not surprising because of the interpretation of $\cX_{\G}$ as a fixed point stack.
		
		Notice that in place of $g$ one could have chosen $g\circ z:x\rightarrow x$ for any $z$ in the center $\Zen(\G)$ of $\G$. When $\Zen(\G)=0$, this shows that the action of $\mathrm{Inn}(\G)$ is trivial, thus we get an induced action of $\Out(\G):=\Aut(\G)/\Inn(\G)$ on $\cX_{\G}$. Recall that $\Out(\G)$ is representable, see for instance~\cite[Exposé 24, Theorem 1.3]{SGA3}.
	\end{remark}
	\subsection{Comparison Properties}
	We start by comparing the stabilizer groups of $\cX_{\G}$ with those of $\cX^{\G}$. The proof of the following lemma is straightforward.
	\begin{lemma}\label{lem: Aut comparison PhiG}
		Let $T$ be a scheme and $(x,\rho)\in\cX_{\G}(T)$. Then:
		\begin{enumerate}
			\item The group $\Aut_{\cX_{\G}}((x,\rho))$ is equal to the centralizer $\Zen(\G,\Aut_{\cX}(x))$ of $\G$ in $\Aut_{\cX}(x)$. In particular, for every $g\in\G$, we have that $\rho(g)\in\Aut_{\cX}(x)$ is in the image of $\Phi_{\G}$ if and only if $g\in\Zen(\G)$.
			\item If $\G=\Aut_{\cX}(x)$, then $\Aut_{\cX}((x,\rho))=\Zen(\G)$.
		\end{enumerate}
	\end{lemma}
	We will see that $\Phi_{\G}$ is particularly well behaved over the locus $\cX^{=\G}$ in $\cX^{\G}$ where the automorphism group is equal to $\G$, which we show to be open in $\cX^{\G}$.
	\begin{lemma}
		The locus of field-valued points $x$ of $\cX^{\G}$ where $\Aut_{\cX}(x)=\G$ is open in $\cX^{\G}$, and defines an open substack $\cX^{=\G}$.
	\end{lemma}
	\begin{proof}
		This statement is étale local on $\cX$, hence we may assume that $\cX=[U/\H]$ for a finite linearly reductive group $\H$ over $k$ acting on a scheme $U$. Recall that $\cX^{\G}=[U^{\G}/\H]$, with $U^{\G}=\cup_{\varphi}U^{\varphi}$, where we are using the notation of the proof of Lemma~\ref{lem: stack of G objects no action properties}.
		Let, $U^{=\varphi}$ be the complement in $U^{\varphi}$ of the union of the closed loci of fixed points under actions by subgroups of $\H$ that do not lie in the image of $\varphi$. By Lemma~\ref{lem: stack of G objects no action properties}, $U^{=\varphi}$ is open in $U^{\varphi}$. Then, $U^{=\G}:=\cup_{\varphi}U^{=\varphi}$ is open in $U^{\G}$. Since $\cX^{=G}=[U^{=\G}/\H]$, we get that $\cX^{=G}\subset\cX^{\G}$ is an open immersion.
	\end{proof}
	\begin{definition}\label{def: open of X lower G}
		We denote by $\cX_{=\G}$ the fiber product $\cX_{\G}\times_{\cX^{\G}}\cX^{=\G}$, which is then an open substack of $\cX_{\G}$ whose field valued points are pairs $(x,\rho)$ such that $\rho:\G\rightarrow\Aut_{\cX}(x)$ is an isomorphism. We denote by $\Phi_{=\G}:\cX_{=\G}\rightarrow\cX^{=\G}$ the map induced by $\Phi_{\G}$.
		Moreover, if $\psi\in\Aut(\G)$, then $\Phi_{\psi}$ preserves $\cX_{=\G}$, hence we can consider the restriction $\Phi_{=\psi}$ of $\Phi_{\psi}$ to that locus.
	\end{definition}
	\begin{proposition}\label{prop: generic gerbe or isom}
		The morphism
		\[
		\begin{tikzcd}
			\Phi_{=\G}:\cX_{=\G}\arrow[r] & \cX^{=\G}
		\end{tikzcd}
		\]
		is an $\Aut(\G)$-torsor, and the coarse moduli space morphism
		\[
		\begin{tikzcd}
			\pi_{\cX^{=\G}}:\cX^{=\G}\arrow[r] & X^{=\G}
		\end{tikzcd}
		\]
		is étale with fibers $B\G$. The composite
		\[
		\begin{tikzcd}
			\pi_{\cX^{=\G}}\circ\Phi_{=\G}:\cX_{=\G}\arrow[r] & X^{=\G}
		\end{tikzcd}
		\]
		is the composite of a banded $\Zen(\G)$-gerbe and an $\Out(\G)$-torsor.
	\end{proposition}
	\begin{proof}
		In Definition~\ref{def: action of AutG} we have constructed an action of $\Aut(\G)$ on $\cX_{\G}$ that leaves $\Phi_{\G}$ invariant, and it clearly preserves $\cX_{=\G}$. To show that $\Phi_{=\G}$ is an $\Aut(\G)$-torsor, we can work étale-locally on $\cX^{=G}$. Consider the base change of $\Phi_{=\G}$ along a map $S\rightarrow\cX^{=\G}$ with $S$ a scheme, thus corresponding to an object $x\in\cX(S)$. A $T$-object of the fiber product is then the data of a map $\phi:T\rightarrow S$, an object $(t,\rho)\in\cX_{=\G}(T)$, and an isomorphism $\alpha:t\rightarrow\phi^*(x)$ in $\cX(T)$. An arrow between two such objects is an arrow in $\cX_{=\G}$ compatible with the isomorphisms as $\alpha$. Therefore, the fiber product is equivalent to the algebraic space whose $T$-objects are locally pairs $(t,\rho)\in\cX_{=\G}(T)$, where $\rho:\G\rightarrow\Aut_{\cX}(x)$ is an isomorphism. After passing to an étale cover $S'\rightarrow S$, there exists an isomorphism $G_{S'}\simeq\uAut_{\cX}(x|_{S'})$, hence the fibered product $\cX_{=\G}\times_{\cX^{=\G}}S'$ is isomorphic to $\Aut(\G)\times S'$, with $\Aut(\G)$ acting on itself by multiplication. This concludes the first part.
		
		The second part follows from the first and third. Alternatively, by~\cite[Corollary 4.4.13]{Alp25} the map $\pi_{\cX^{=\G}}$ is locally on $X^{=\G}$ of the form $[\Spec A/\G]\rightarrow\Spec A^{\G}$, and by assumption $\G$ acts trivially and $\Spec A^{\G}=\Spec A$.
		
		Now, notice that $\Zen(\G)$ defines a 2-group-structure on $\cX_{=G}$, see~\cite[Section 5]{ACV03} for the notion of 2-group-structure. Indeed, by Lemma~\ref{lem: Aut comparison PhiG}, for every object $(x,\rho)$ of $\cX_{=G}$ we have a canonical identification $\Aut_{\cX_{\G}}((x,\rho))\simeq\Zen(G)$ induced by $\rho$ itself, hence it is preserved by morphism in $\cX_{=\G}$, as they are $\G$-equivariant.
		This defines the $Z(\G)$-2-group-structure on $\cX_{=G}$, thus inducing a $Z(\G)$-rigidification $\cX_{=G}\rightarrow\widetilde{\cX}_{=\G}$; see~\cite[Section 5]{ACV03} or~\cite[Appendix A]{AOV08} for the rigidification construction.
		Moreover, by Lemma~\ref{lem: Aut comparison PhiG}, $\widetilde{\cX}_{=\G}$ is equivalent to an algebraic space, and objects differing by an inner automorphism of $\G$ are identified. In particular, there is a residue action of $\Out(\G)=\Aut(\G)/\Inn(\G)$ on $\widetilde{\cX}_{=\G}$, and the induced map $\widetilde{\cX}_{=\G}\rightarrow X^{=\G}$ is $\Out(\G)$-invariant. It follows easily that $\widetilde{\cX}_{=\G}\rightarrow X^{=\G}$ is a $\Out(\G)$-torsor, which concludes the third part.
	\end{proof}
	\begin{corollary}\label{cor: generic gerbe or isom complete case}
		The following holds.
		\begin{enumerate}
			\item Suppose that $\G$ is adjoint, that is, $\Aut(\G)=\Inn(\G)$. Then,
			\[
			\begin{tikzcd}
				\Phi_{=\G}:\cX_{=\G}\arrow[r] & \cX^{=\G}
			\end{tikzcd}
			\]
			is an $\Inn(\G)$-torsor, and the composite
			\[
			\begin{tikzcd}
				\pi_{\cX^{=\G}}\circ\Phi_{=\G}:\cX_{=\G}\arrow[r] & X^{=\G}
			\end{tikzcd}
			\]
			is a banded $\Zen(\G)$-gerbe.
			\item Suppose $\Zen(\G)=0$. Then, $\cX_{=G}$ is an algebraic space and the composite
			\[
			\begin{tikzcd}
				\pi_{\cX^{=\G}}\circ\Phi_{=\G}:\cX_{=\G}\arrow[r] & X^{=\G}
			\end{tikzcd}
			\]
			is an $\Out(\G)$-torsor.
			\item Suppose that $\G$ is complete, that is, adjoint with trivial center. Then,
			\[
			\begin{tikzcd}
				\Phi_{=\G}:\cX_{=\G}\arrow[r] & \cX^{=\G}
			\end{tikzcd}
			\]
			is a trivial $\G$-torsor, and the composite
			\[
			\begin{tikzcd}
				\pi_{\cX^{=\G}}\circ\Phi_{=\G}:\cX_{=\G}\arrow[r] & X^{=\G}
			\end{tikzcd}
			\]
			is an isomorphism. In particular, $\pi_{\cX^{=\G}}:\cX^{=\G}\rightarrow X^{=\G}$ is a trivial gerbe.
		\end{enumerate}
	\end{corollary}
	\begin{proof}
		It follows immediately from Proposition~\ref{prop: generic gerbe or isom} and its proof.
	\end{proof}
	\begin{remark}\label{rmk: gerbe complete groups}
		The previous corollary implies that if $\G$ is a finite complete group then $\pi_{\cX^\G}$ is generically a trivial gerbe. This fits in the more general framework of non-abelian bands, or liens, developed by Giraud in~\cite{Gir65}; see~\cite[Section 3]{EHKV01} for a quicker introduction. If $\G\rightarrow Y$ is an fppf group scheme over a scheme $Y$, a $\G$-gerbe is a morphism $\cG\rightarrow Y$ such that there exists an fppf cover $Y'\rightarrow Y$ such that $\cG\times_YY'\simeq B\G\times_YY'$. We say that $\cG\rightarrow Y$ is trivial if $\cG=B\G$. There are two cohomological obstructions to the triviality of such gerbes, the first of which lives in $\H^1(Y,\Out(\G))$, thus it corresponds to a $\Out(\G)$-torsor $P\rightarrow Y$. If this vanishes, then the second obstruction is a class in $\H^2(Y,\Zen(\G))$; in particular, there is always such a cohomology class in $\H^2(P,\Zen(\G))$, hence a corresponding $\Zen(\G)$-banded gerbe $\cW\rightarrow P$. Then, Proposition~\ref{prop: generic gerbe or isom} yields a geometric realization of these obstructions, given by the composite $\cX_{=\G}\rightarrow X^{=\G}$.
	\end{remark}
	\begin{remark}\label{rem: canonical desingularization}
		When $\cX$ is smooth and $\cX_{=\G}$ is connected with closure $\overline{\cX_{=\G}}$ in $\cX_{\G}$, the morphism $[\overline{\cX_{=\G}}/\Aut(\G)]\rightarrow\cX^{\G}$ is a canonical desingularization of $\cX^{\G}$ in the category of stacks, by Propositions~\ref{prop: basic facts about X{G}} and~\ref{prop: generic gerbe or isom}. Similarly, $[\widetilde{\cX}_{\G}/\Out(\G)]\rightarrow X^{\G}$ is a canonical desingularization of $X^{\G}$ in the category of stacks, where $\overline{\cX_{=\G}}\rightarrow\widetilde{\cX}_{\G}$ is the $\Zen(\G)$-rigidification.
	\end{remark}
	In the next sections we will apply this construction to study monodromy groups of some étale covers, in particular when $\G$ is complete. Moreover, it is sometimes easier to find a presentation as a quotient stack for $\cX_{\G}$ rather than for $\cX^{\G}$, and this can be used to compute the monodromy more directly. We illustrate the above theory with one example that will be useful for Section~\ref{sec: monodromy G cubic surfaces} as well.
	\begin{example}\label{exm: G cubic surfaces}
		For simplicity, let us assume the base field to be $\C$. Recall that $\cM$ denotes the stack parametrizing smooth cubic surfaces, with coarse moduli space $\pi_{\cM}:\cM\rightarrow\M$. Then, for every group $\G$ realizable as the automorphism group of some smooth cubic surface, see~\cite[Table 9.6]{Dol12} and~\cite[Table 5]{BLP23}, we can consider the closed substack $\cM^{\G}$ of $\cM$. In this case, the stack $\cM_{\G}$ parametrizes pairs $(S\rightarrow B,\rho:\G\hookrightarrow\Aut_B(S))$ of a family of smooth cubic surfaces $S\rightarrow B$ and a faithful $\G$-action on it; the map $\Phi_{\G}:\cM_{\G}\rightarrow\cM^{\G}$ forgets the $\G$-action. Let $\cF:\cL\rightarrow\cM$ be the cover given by the 27 lines (see Subsection~\ref{subsec: lines on cubics generic case}), and $\cF^{\G}:\cL^{\G}\rightarrow\cM^{\G}$ the restriction to $\cM^{\G}$. Then, there is an induced cover $\L^{\G}\rightarrow M^{\G}$, which is finite and étale over $\M^{=\G}$. This induces a commutative diagram with cartesian squares
		\begin{equation}\label{diag: covers cubic surfaces}
			\begin{tikzcd}
				\cL_{\G}\arrow[d,"\phi"]\arrow[dd,bend right=40,"\cF_{\G}"']\arrow[r] & \cL^{\G}\arrow[d,"\psi"]\arrow[rd]\\
				\hat{\cL}_{\G}\arrow[d]\arrow[r] & \hat{\cL}^{\G}\arrow[r]\arrow[d] & \L^{\G}\arrow[d,"\F^{\G}"]\\
				\cM_{\G}\arrow[r,"\Phi_{\G}"] & \cM^{\G}\arrow[r,"\pi_{\cM^{\G}}"] & \M^{\G},
			\end{tikzcd}
		\end{equation}
		where we are defining some of the stacks by this diagram.
		Here, a $B$-object of $\cL_{\G}$ is a pair $(L\subset S\xrightarrow{f}B,\rho)$ of a family of smooth cubic surfaces with a line $L$ and a faithful action $\rho:\G\hookrightarrow\Aut(S)$; an arrow is a $\G$-equivariant morphism of families preserving the line. Moreover, $\psi:\cL^{\G}\rightarrow\hat{\cL}^{\G}$ is generically a homeomorphism, while $\hat{\cL}^{\G}$ is generically FPR. Even though the spaces might not be connected, the construction still works and respects the connected components. From now on, we assume $\cM_{=\G}$ to be connected, hence irreducible being smooth; thus, the same holds for $\cM^{=\G}$ and $\M^{=\G}$.
		
		Recall that it can happen that $\cM^{\G}$ and $\M^{\G}$ are non-normal. As $\cF^{\G}:\cL^{\G}\rightarrow\cM^{\G}$ is finite and étale, its monodromy group is defined anyway, but it is not clear what $\Mon_{\F^{\G}}$ should be. To make sense of it, we will consider the restriction $\F^{=\G}$, as it is finite and étale by Lemma~\ref{lem: factorization FPR and gerbe}.
		
		Moreover, the monodromy of $\cF^{\G}$ can change if we restrict to some open substack, for instance to $\cM^{=\G}$, as we will see in Section~\ref{sec: monodromy G cubic surfaces}. For example, the monodromy group of $\cF^{\G}$ with $\G=\mC_2$ is clearly the full $\W(E_6)$, as this is generated by automorphism groups of some special cubics, which are all points in $\cM^{\G}$; see Proposition~\ref{prop: monodromy general case}. Arguably, this group is too big to be what we are looking for, and it looses the connection with the solvability of the problem of finding the lines. On the other hand, working with coarse moduli spaces, the monodromy group of $\F^{=\G}$ does not see the generic stabilizer, which is something we probably want to keep track of. A good compromise seems to be computing $\Mon_{\cF^{=\G}}$, which is the one we are going to study; when possible, we will compare it with the other choices of monodromy groups.
		
		Over $\cM^{=\G}$, $\psi$ is a Galois cover with Galois group $\Mon_{\psi}=\G$, since every automorphism of a smooth cubic surface acts faithfully on the set of lines. From this, Proposition~\ref{prop: monotonous gerbe case}, and the last part of Lemma~\ref{lem: factorization FPR and gerbe} we get an exact sequence
		\begin{equation}\label{eq: exact monodromy G cubic surfaces}
			\begin{tikzcd}
				0\arrow[r] & \G\arrow[r] & \Mon_{\cF^{=\G}}\arrow[r] & \Mon_{\F^{=\G}}\arrow[r] & 0.
			\end{tikzcd}
		\end{equation}
		We would like to compute $\Mon_{\F^{=\G}}$ and understand when~\eqref{eq: exact monodromy G cubic surfaces} splits. To get a splitting as direct sum, we would need to find a $\G$-cover $g:\cW\rightarrow\cM^{=\G}$ such that the monodromy group of the pullback $h:\cL^{=\G}\times_{\cM^{=\G}}\cW\rightarrow\cW$ is isomorphic to $\Mon_{\F^{=\G}}$ under $\Mon_h\hookrightarrow\Mon_{\cF^{=\G}}\twoheadrightarrow\Mon_{\F^{=\G}}$. The natural candidate is $\cW=\cM_{=\G}$.
		
		Now, over $\cM_{=\G}$ the cover $\phi$ is Galois with Galois group $Z(\G)$, even though it is not a homeomorphism. This follows essentially by Lemma~\ref{lem: Aut comparison PhiG} and the discussion above.
		Since $\hat{\cL}_{\G}\rightarrow\cM_{\G}$ is FPR, if $\Out(\G)=0$ then its monodromy group is isomorphic to $\Mon_{\F^{=\G}}$ via the natural inclusion, thanks to Corollary~\ref{cor: generic gerbe or isom complete case}. This shows that, if $\Out(\G)=0$, the exact sequence
		\[
		\begin{tikzcd}
			0\arrow[r] & \G/Z(\G)\arrow[r] & (\Mon_{\cF^{=\G}})/Z(\G)\arrow[r] & \Mon_{\F^{=\G}}\arrow[r] & 0
		\end{tikzcd}
		\]
		splits as a direct sum. In particular, if $\G$ is complete, then the exact sequence~\eqref{eq: exact monodromy G cubic surfaces} splits, hence we get
		\begin{equation}\label{eq: split monodromy G complete}
			\Mon_{\cF^{=\G}}\simeq\G\times\Mon_{\cF_{=\G}}.
		\end{equation}
		This is true more generally whenever $\pi_{\cM^{=\G}}$ is a trivial gerbe. In general, the sequence~\eqref{eq: exact monodromy G cubic surfaces} does not split, not even as a semidirect product, as we will see for the case $\G=\mC_2$ in Subsection~\ref{subsec: mu2 cubic surfaces}.
	\end{example}
	The above argument generalizes to give the following natural constraint on monodromy groups of symmetric problems.
	\begin{corollary}\label{cor: inclusion in normalizer}
		Assume $\cX$ and $\cX^{=\G}$ to be connected, and let $\cF:\cY\rightarrow\cX$ be a representable, finite, étale cover, with restriction $\cF^{=\G}:\cY^{=\G}\rightarrow\cX^{=\G}$. Let $\F:Y\rightarrow X$ and $\F^{=\G}:Y^{=\G}\rightarrow X^{=\G}$ be the induced covers between moduli spaces. Suppose that for some geometric point $x$ of $\cX^{=\G}$ the composite
		\[
		\begin{tikzcd}
			\G\simeq\pi_1^h(\cX,x)\arrow[r,"\omega_x"] & \pi_1(\cX,x)\arrow[r,twoheadrightarrow] & \Mon_{\cF}
		\end{tikzcd}
		\]
		is injective, that is, $\G$ acts faithfully on the fibers of $\cF$. Then, there is a short exact sequence
		\begin{equation}\label{eq: exact sequence monodromy X=G}
			\begin{tikzcd}
				0\arrow[r] & \G\arrow[r,"\omega_x"] & \Mon_{\cF^{=\G}}\arrow[r] & \Mon_{\F^{=\G}}\arrow[r] & 0.
			\end{tikzcd}
		\end{equation}
		In particular, $\Mon_{\cF^{=\G}}$ is contained in the normalizer $\Nor(\G,\Mon_{\cF})$ of $\G$ in $\Mon_{\cF}$, and $\Mon_{\F^{=\G}}\subseteq\Nor(\G,\Mon_{\cF})/\G$. If moreover $\cX_{=\G}$ is connected, then the monodromy of the induced cover $\cF_{=\G}:\cY_{=\G}\rightarrow\cX_{=\G}$ is contained in the centralizer $\Zen(\G,\Mon_{\cF})$. Finally, if $\G$ is complete then the exact sequence~\eqref{eq: exact sequence monodromy X=G} splits as a direct sum.
	\end{corollary}
	\begin{proof}
		By Proposition~\ref{prop: generic gerbe or isom} and the end of Lemma~\ref{lem: factorization FPR and gerbe}, we have the exact sequence of the statement, where the first morphism is injective by hypothesis. The existence of the exact sequence implies that $\G$ is normal in $\Mon_{\cF^{=\G}}$, hence the first two containment. If $\G$ is complete, Corollary~\ref{cor: generic gerbe or isom complete case} implies that $\cX_{=\G}\rightarrow X^{=\G}$ is an isomorphism, yielding a section to $\pi_{=\cX}$ and an isomorphism $\Mon_{\F^{=\G}}\simeq\Mon_{\cF_{=\G}}$. Consequently, the inclusion $\Mon_{\cF_{=\G}}\hookrightarrow\Mon_{\F^{=\G}}$ induces a splitting of~\eqref{eq: exact sequence monodromy X=G}. As $\Phi_{=\G}:\cX_{=\G}\rightarrow\cX^{=\G}$ is a $\G$-torsor, the image of the splitting is normal with quotient $\G$, showing that $\Mon_{\cF^{=\G}}\simeq\G\times\Mon_{\F^{=\G}}$.
		
		The inclusion $\Mon_{\cF_{=\G}}\subseteq\Zen(\G,\Mon_{\cF})$ follows from requiring that the monodromy action over a point $(x,\rho)$ preserves the faithful action $\rho$. Alternatively, recall that by Proposition~\ref{prop: generic gerbe or isom} the morphism $\Phi_{=\G}:\cX_{=\G}\rightarrow\cX^{=\G}$ is a $\Aut(\G)$-torsor. This induces an exact sequence
		\[
		\begin{tikzcd}
			0\arrow[r] & \Mon_{\cF_{=\G}}\arrow[r] & \Mon_{\cF^{=\G}}\arrow[r] & \Aut(\G)
		\end{tikzcd}
		\]
		where the last morphism is the composite
		\[
		\begin{tikzcd}
			\Mon_{\cF^{=\G}}\arrow[r,hookrightarrow] & \Nor(\G,\Mon_{\cF^{=\G}})\arrow[r] & \Aut(\G);
		\end{tikzcd}
		\]
		the first inclusion was proved above, and the second morphism sends an element $h\in\Nor(\G,\Mon_{\cF^{=\G}})$ to the automorphism $g\mapsto hgh^{-1}$ of $\G$. As the kernel of the last map is the centralizer $\Zen(\G,\Mon_{\cF})$, this concludes.
	\end{proof}
	\begin{remark}\label{rmk: compatible isom complete}
		The fact that $\Mon_{\cF_{=G}}\simeq\Mon_{\F^{=\G}}$ whenever $\G$ is complete is compatible with the property of complete groups whereby $\Zen(\G,\Mon_{\cF})\simeq\Nor(\G,\Mon_{\cF})$ under the natural projection.
	\end{remark}
	An important question is whether the inclusions in Corollary~\ref{cor: inclusion in normalizer} can be reversed. This is not always the case, even in `nice' situations, as shown by the next example. In Section~\ref{sec: monodromy G cubic surfaces} we present examples where the inclusions of Corollary~\ref{cor: inclusion in normalizer} are equalities.
	\begin{example}\label{exm: counterxample equality conjecture}
		Consider the action of the symmetric group $\fS_3$ on $\P_{\C}^1$ induced by an injection $\fS_3\hookrightarrow\PGL_2$ with image the subgroup generated by the matrices
		\[
		\sigma=\begin{bmatrix}
			0 & 1\\
			1 & 0
		\end{bmatrix},\qquad\rho=\begin{bmatrix}
			1 & 0\\
			0 & \zeta_3
		\end{bmatrix},
		\]
		where $\zeta_3\not=1$ is a third root of unity. This yields a Galois cover
		\[
		\begin{tikzcd}
			\cY:=\P^1\arrow[r,"f"] & {[\P^1/\fS_3]}=:\cX.
		\end{tikzcd}
		\]
		The fixed points of $\mC_3=\langle\rho\rangle$ are $[1:0]$ and $[0:1]$, which are swapped by $\sigma$. It follows that $\cX^{=\G}\simeq B\mC_3$ and it is connected, hence $\Mon_{f^{=\G}}\simeq\mC_3\subsetneq\fS_3=\Nor(\mC_3,\fS_3)$. This gives an example where the cover is Galois between smooth, proper stacks with schematic source, but the containment of Corollary~\ref{cor: inclusion in normalizer} are strict. In constrast, $\cX_{=\mC_2}=\cX^{=\mC2}\simeq B\mC_2\sqcup B\mC_2$, which is disconnected, and over each connected component the cover has monodromy group equal to $\Zen(\mC_2,\Mon_f)=\Nor(\mC_2,\Mon_f)$.
	\end{example}
	
	As we will see shortly, the stack $\cX_{\G}$ is usually not connected nor equidimensional, even when $X^{\G}$ is. In Section~\ref{subsec: S3xmu2 cubic surfaces} we give an example where also $\cX_{=\G}$ is disconnected.
	In the following subsection we study the connected components in the case of smooth cubic surfaces.
	
	\subsection{Connected Components in the Case of Cubic Surfaces}\label{subsec: G cubic surfaces with fixed representation}
	Let $S$ be a smooth cubic surface over $\C$, and set $\V:=\H^0(\omega_{S}^{\vee})$. Recall that $S$ is anti-canonically embedded in $\P\V\simeq\P^3$, where it is defined as the zero locus of a degree 3 polynomial $f\in\H^0(\cO_{\P\V}(3))$. In particular, any automorphism $\psi$ of $S$ extends to an automorphism $\overline{\psi}$ of $\P^3$, which is induced by the automorphism $\widetilde{\psi}:\H^0(\omega_{S}^{\vee})\rightarrow\H^0(\omega_{S}^{\vee})\simeq\C^{\oplus4}$. Thus, given an injection $\rho:\G\hookrightarrow\Aut(S)$, we get a natural four dimensional $\G$-representation $\V$, such that the zero locus $S$ of $f\in\H^0(\cO_{\P\V}(3))$ is preserved. This means that there exists a character $\chi$ of $\G$ such that $g\cdot f=\chi(g)\cdot f$ for every $g\in\G$; notice that this character only depends on the $\G$-action. Then, we can define an immersion $S\rightarrow\P\V\xrightarrow{\simeq}\P(\V\otimes\chi^{-1})$, so that $g\cdot f=f$. Notice that all the immersion are $\G$-equivariant as we are working with projective spaces, so the characters play a role only when talking about $\cO_{\P\V}(1)$ and its tensor powers. All of this works for families of $\G$-cubic surfaces.
	\begin{definition}\label{def: G cubic surfaces with fixed representation}
		Let $\V$ be a 4-dimensional representation of $\G$. We denote by $\cM_{\G}^{\V}$ be the full subcategory of $\cM_{\G}$ parametrizing objects $(f:S\rightarrow B,\rho)$ of $\cM_{\G}$ whose induced $\G$-representation $\H^0(\omega_{S_b}^{\vee})$ is isomorphic to $\V$ for every $b\in B(\C)$.
	\end{definition}
	\begin{lemma}\label{lem: cubic surface representation is locally constant}
		The category $\cM_{\G}^{\V}$ is an open and closed substack of $\cM_{\G}$.
	\end{lemma}
	\begin{proof}
		We know that for every flat family of smooth cubic surfaces $f:S\rightarrow B$, the pushforward $E:=f_*\omega_{S/B}^{\vee}$ is a rank 4 vector bundle satisfying base change, with $\G$ acting on the fibers over $B$. One way to see this is to notice that $\cM$ is smooth and $\dim\H^0(\omega_{S_0})=4$ for every smooth cubic surface $S_0$, hence one can apply Grauert's Theorem to the universal family $\cS\rightarrow\cM$. Now, assume $B$ to be connected; then, we want to show that for every geometric point $b$ of $B$ the $G$-representation $E_b$ does not depend on $b$. This follows immediately from the fact that $\G$ is linearly reductive.
	\end{proof}
	We will be interested in considering only projectivization of representations, as we are interested in the induced action of $\G$ on $\P\V=\P(\H^0(S,\omega_{S}^{\vee}))\simeq\P^3$ rather than the action on $\V=\H^0(S,\omega_{S}^{\vee})$. Two actions on $\V$ inducing the same one on $\P\V$ differ by a character. Moreover, we will see that in general not every action on $\V$ is possible. This justifies the introduction of the following substack of $\cM_{\G}$.
	\begin{definition}
		Let $\V$ be a 4-dimensional representation of $\G$, with associated projective representation $\P\V$. We denote by $\cM_{\G}^{\P\V}$ be the full subcategory of $\cM_{\G}$ parametrizing objects $(f:S\rightarrow B,\rho)$ of $\cM_{\G}$ whose induced projective $\G$-representation $\P(\H^0(\omega_{S_b}^{\vee}))$ is isomorphic to $\P\V$ for every $b\in B(\C)$.
	\end{definition}
	The following is an analogue and direct consequence of Lemma~\ref{lem: cubic surface representation is locally constant}.
	\begin{lemma}\label{lem: cubic surface projective representation is locally constant}
		The category $\cM_{\G}^{\P\V}$ is an open and closed substack of $\cM_{\G}$.
	\end{lemma}
	\section{A Concrete Example: Monodromy of $\G$-Cubic Surfaces}\label{sec: monodromy G cubic surfaces}
	In this section we compute the monodromy group of $\cF^{=\G}:\cL^{=\G}\rightarrow\cM^{=\G}$ and $\F^{=\G}:\L^{=\G}\rightarrow\M^{=\G}$ for $\G$ equal to $\fS_4$, $\fS_3$, $\fS_3\times\mC_2$, and $\mC_2$. Here, $\fS_n$ is the $n$-th symmetric group, while $\mC_n$ is the cyclic group of order $n$. Each of these examples illustrates a different technique, all carried out within the framework of stacks.
	\begin{remark}\label{rmk: irreducibility of MG}
		For every group $\G$ that we consider, the moduli stack $\cM^{\G}$ and its moduli space $\M^{\G}$ are irreducible, thanks to the results in~\cite{Tu01}. In particular, also the opens $\cM^{=\G}$ and $\M^{=\G}$ are irreducible, and the monodromy groups of $\cF^{=\G}$ and $\F^{=\G}$ are well defined. See also Example~\ref{exm: G cubic surfaces}.
	\end{remark}
	\subsection{Automorphism Groups and Eckardt Points}
	As it is probably already clear from what we have done so far, it is important to understand the automorphism group of cubic surfaces and how it relates to the geometry of these surfaces. It is a classical fact that the automorphism group of cubic surfaces is closely related to the number and configurations of Eckardt points; in this subsection we briefly recall the main results on this matter. See~\cite[Chapter 9]{Dol12} for more details.
	\begin{definition}\label{def: Eckardt points}
		Let $S$ be a smooth cubic surface. A point of $S$ is called an Eckardt point if it is equal to the intersection of three different lines contained in $S$.
	\end{definition}
	It is a classical fact that a smooth cubic surface over $\C$ has at most 18 Eckardt points. This bound is realized by the Fermat cubic, defined as the zero locus of $x^3+y^3+z^3+w^3$ in $\P^3$; it is the only isomorphism class of cubic surfaces with 18 Eckardt points.
	\begin{lemma}\label{lem: facts on Eckardt points}
		Let $e$ be an Eckardt point of a smooth cubic surface $S$ anti-canonically embedded in $\P^3$.
		\begin{enumerate}
			\item There exists a unique involution of $\P^3$ preserving $S$ and with fixed point locus equal to the disjoint union of an hyperplane and the point $e$. Conversely, given an involution of $\P^3$ preserving $S$ and with fixed point locus equal to the disjoint union of a hyperplane and a point $p\in S$, we have that $p$ is an Eckardt point of $S$.
			\item Suppose that $e'$ is another Eckardt point, and let $E$ be the line passing through $e$ and $e'$. Then, either $M\subset S$ or $M\cap S=\{e,e',e''\}$ with $e''$ a different Eckardt point, and the two possibilities exclude each other.
			\item In the first case of the previous point, the involutions defined by $e$ and $e'$ commute, and fix both $e$ and $e'$. In the second case, the involutions defined by $e$, $e'$ and $e''$ generate a subgroup of $\Aut(S)$ isomorphic to $\fS_3$.
		\end{enumerate}
	\end{lemma}
	\begin{proof}
		See~\cite[Lemmas 9.1.13, 9.1.14, 9.1.15]{Dol12}.
	\end{proof}
	In~\cite[Section 9.5.3]{Dol12} and~\cite[Table 5]{BLP23} the authors have classified the automorphism groups of smooth cubic surfaces, relating them to the number and configuration of the Eckardt points. For example, one can check that a cubic surface admits non-trivial automorphisms if and only if it contains an Eckardt point.
	
	It is possible to define and study the moduli stacks of cubic surfaces with marked Eckardt points, and relate them to the stacks $\cM_{\G}$. An analysis of those stacks is beyond the scope of this work, hence omitted.
	\subsection{$\fS_4$-Cubic Surfaces}\label{subsec: S4 cubic surfaces}
	We use the notation of Example~\ref{exm: G cubic surfaces}, with $\G=\fS_4$.
	
	In this section we compute the monodromy group $\Mon_{\F}$ of the étale cover $\cF^{=\fS_4}:\cL^{=\fS_4}\rightarrow\cM^{=\fS_4}$. By Remark~\ref{rmk: irreducibility of MG}, we know that $\cM^{=\fS_4}$ is connected. As $\fS_4$ is complete, Corollary~\ref{cor: inclusion in normalizer} shows that $\Mon_{\cF^{=\fS_4}}=\fS_4\times\Mon_{\F^{=\fS_4}}$. Therefore, it is enough to compute $\Mon_{\F^{=\fS_4}}$, which is also of independent interest. In fact, $\Mon_{\F^{=S_4}}$ has already been computed by Brazelton and Raman in~\cite{BR24}, showing that it is isomorphic to $\mC_2\times\mC_2$, using interesting methods from Hodge theory; our argument is different from theirs.
	
	Since $\fS_4$ is a complete group, Corollary~\ref{cor: generic gerbe or isom complete case} shows that over the open subscheme $\M^{=\fS_4}$ the composite
	\[
	\begin{tikzcd}
		\cM_{=\fS_4}\arrow[r,"\Phi_{\fS_4}"'] & \cM^{=\fS_4}\arrow[r,"\pi_{\cM^{\fS_4}}"'] & \M^{=\fS_4}
	\end{tikzcd}
	\]
	is an isomorphism; in particular, $\cM_{=\fS_4}$ is also connected. Therefore, it is enough to compute the monodromy group of $\cF_{\fS_4}:\cL_{\fS_4}\rightarrow\cM_{\fS_4}$ after restricting to the irreducible component of $\cM_{\fS_4}$ containing $\cM_{=\fS_4}$.
	Moreover, by Corollary~\ref{cor: inclusion in normalizer}, we know that
	\begin{align*}
		\fS_4\subset\Mon_{\cF^{=\fS_4}}\subseteq \Nor(\fS_4,W(E_6))\simeq\fS_4\times\mC_2\times\mC_2,\\
		\Mon_{\F^{=S_4}}\subseteq\Nor(\fS_4,W(E_6))/\fS_4\simeq\mC_2\times\mC_2,
	\end{align*}
	where the two isomorphism are simple computations; see also~\cite[Proposition 5.3]{BR24}. Then, we are left with generating enough elements in the monodromy group; unfortunately, we will not be able to use automorphism groups to do so. See~\cite[Lemma 5.15]{BR24} for a computational way of constructing monodromy, using certified tracking algorithms.
	
	In our case, we will directly compute the Galois group by first finding a nice presentation of $\cM_{=\fS_4}$ as a quotient stack and then solving for the lines, thus computing in which field extension their equations live. Notice that we already know this method to succeed as the group is contained in $\mC_2\times\mC_2$, hence solvable.
	
	We start with a remark about the geometry of $\fS_4$-cubic surfaces.
	\begin{remark}\label{rmk: geometry S4 surfaces}
		Let $\fS_4$ act on a smooth surface $S$ with exactly 6 Eckardt points. Then, $\Aut(S)\simeq\fS_4$ is generated by the involutions associated to the Eckardt points. It is a classical fact that all 6 Eckardt points lie on a plane $H$ with the disposition as in Figure~\ref{fig:6EckardtPoints}; see also~\cite[Subsection 6.5]{BLP23}. In particular, by the commutativity requirements between the Eckardt involutions in Lemma~\ref{lem: facts on Eckardt points}, it is easy to see that those correspond exactly to the simple transpositions in $\fS_4$. Moreover, $\fS_4$ fixes one point not contained in $H$, as $\fS_4$ is generated by three Eckardt involutions whose planes of fixed points intersect at one point. In general, given a faithful $\fS_4$-action on any smooth cubic surface $\S$, this is induced by 6 Eckardt points if and only if it is of the form above, with every simple transposition being involutions of $\P^3$ with fixed locus consisting of a plane and an isolated point lying on the surface.
		
		Using again Lemma~\ref{lem: facts on Eckardt points}, we see that the induced representation up to character is $\W_4:=\mathds{1}\oplus\V_3$, where $\V_3$ is the standard representation of $\fS_4$, that is, its only irreducible representation of dimension 3. Actually, only one representation is possible, not even up to character, as the proof of the following theorem will show; however, we are not interested in understanding which one it is.
	\end{remark}
	\begin{figure}
		\centering
		\includegraphics[width=0.5\textwidth]{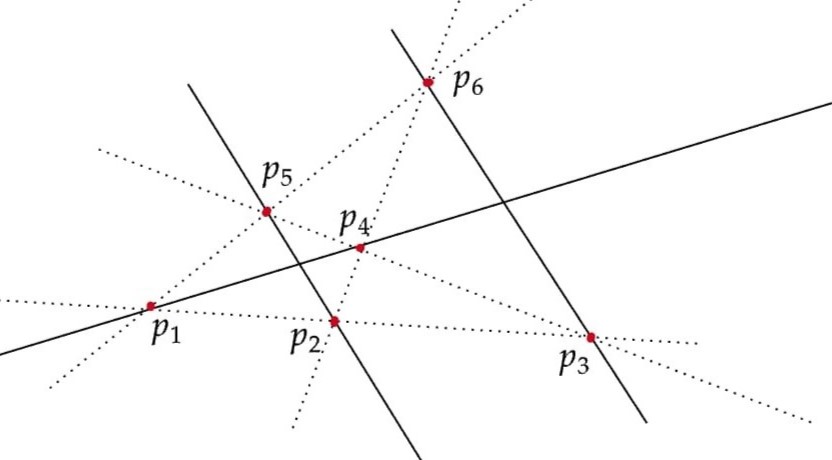} 
		\caption{Configuration of 6 Eckardt points}
		\label{fig:6EckardtPoints}
	\end{figure}
	\begin{theorem}\label{thm: S4 presentation}
		Consider the $\fS_4$-representation $\W_4=\mathds{1}\oplus\V_3$, where $\V_3$ is the 3-dimensional standard $\fS_4$-representation. Then, $\cM_{\fS_4}^{\P\W_4}$ is the closure of $\cM_{=\fS_4}$ in $\cM_{\fS_4}$, and
		\[
		\cM_{\fS_4}^{\P\W_4}\simeq\A^1\setminus\{0\}.
		\]
		Here, every $a\in\A^1\setminus\{0\}$ corresponds to the cubic surface $S_a$ with equation
		\begin{equation}\label{eq: equation S4 cubic surface}
			ax^3+x(y^2+z^2+w^2-yw-zw)+w(y-z)(w-y-z)
		\end{equation}
		and the $\fS_4$ action is obtained by restriction from $\fS_4\hookrightarrow\PGL_4$ given by
		\[
		(12)=\begin{bmatrix}
			1 & 0 & 0 & 0\\
			0 & 1 & 0 & 0\\
			0 & 0 & 1 & 0\\
			0 & 1 & 1 & -1
		\end{bmatrix},\quad (13)=\begin{bmatrix}
			1 & 0 & 0 & 0\\
			0 & -1 & 0 & 1\\
			0 & 0 & 1 & 0\\
			0 & 0 & 0 & -1
		\end{bmatrix},\quad (34)=\begin{bmatrix}
			1 & 0 & 0 & 0\\
			0 & 0 & -1 & 0\\
			0 & -1 & 0 & 0\\
			0 & -1 & -1 & 1
		\end{bmatrix}.
		\]
		Moreover, the Eckardt points inducing the action are
		\begin{align*}
			p_1=[0:0:0:1], && p_2=[0:1:0:0], && p_3&=[0:1:0:1]\\
			p_4=[0:1:1:1], && p_5=[0:0:1:0], && p_6&=[0:0:1:1].
		\end{align*}
	\end{theorem}
	\begin{proof}
		Let $(f:S\rightarrow B,\rho:\fS_4\hookrightarrow\Aut(f))$ be an element of $\cM_{\fS_4}^{\P\W_4}$. By Remark~\ref{rmk: geometry S4 surfaces}, we know that we can describe the action in terms of a configuration of 6 Eckardt points as in Figure~\ref{fig:6EckardtPoints}, so that $(12)$ corresponds to $p_1$, $(13)$ to $p_2$, and $(34)$ to $p_4$. Here, $p_i:B\rightarrow S$ are considered as sections to $f$. Let us work Zariski-locally on $B$, so that we can assume it to be affine, say $\Spec R$ for a ring $R$. Locally on $B$ we have a trivialization of $f_*\omega_{S/B}^{\vee}$, inducing a closed embedding $S\subset\P_R^3=:\P^3$. After applying a change of coordinates, we can assume that
		\begin{align*}
			p_1=[0:0:0:1], && p_2=[0:1:0:0], && p_3&=[0:1:0:1]\\
			p_4=[0:1:1:1], && p_5=[0:0:1:0], && p_6&=[0:0:1:1];
		\end{align*}
		in particular, the last part of the theorem will be automatically satisfied. In the same way, we can assume that the point fixed by $\fS_4$ is $q:=[1:0:0:0]$. We claim that this determines the $\fS_4$-action on $\P^3$, and the induced injection $\fS_4\hookrightarrow\PGL_4$ is the one described in the statement. We show this only for $(12)$, as the other cases are analogous. Recall that $(12)$ fixes $q$, $p_1$ and $p_2$, while it swaps $p_5$ with $p_6$, and $p_2$ with $p_3$. This implies that
		\[
		(12)=\begin{bmatrix}
			1 & 0 & 0 & 0\\
			0 & b & 0 & 0\\
			0 & 0 & b & 0\\
			0 & b & b & -b
		\end{bmatrix}
		\]
		for some $b\not=0$. Imposing that the fixed point locus contains a plane, we get that $b=1$, as wanted.
		
		Now, let us describe the possible polynomials $g=g(x,y,z,w)$ defining $S\subset\P^3$. We already know that $S$ contains the three lines on the plane $x=0$ with equations $w=0$, $y-z=0$ and $w-y-z=0$, respectively. It follows that
		\begin{equation}\label{eq: S4 polynomial first form}
			g(x,y,z,w)=x\cdot g_2(x,y,z,w)+cw(y-z)(w-y-z)
		\end{equation}
		for some $c\in\C\setminus\{0\}$, and $g_2\not=0$ of degree 2. Notice that $w(y-z)(w-y-z)$ is $\fS_4$-invariant, hence $g_2$ is invariant as well, not even up to a character. Therefore,
		\begin{equation}\label{eq: S4 polynomial second form}
			g(x,y,z,w)=ax^3+b(y^2+z^2+w^2-yw-zw)+cw(y-z)(w-y-z)
		\end{equation}
		for $abc\not=0$, as it needs to be smooth. At the moment, we are not interested in completely determining when the polynomial is smooth.
		
		We want to know when two smooth cubic surfaces with equations like~\eqref{eq: S4 polynomial second form} are isomorphic $\fS_4$-equivariantly. As an $\fS_4$-equivariant isomorphism between them would yield an isomorphism between the induced $\fS_4$-representations, the isomorphism extends to an automorphism of $\P^3$ commuting with $\fS_4$, thus corresponding to a matrix in the centralizer $Z(\fS_4,\PGL_4)$ of $\fS_4$ in $\PGL_4$. This can be easily characterized as the group of matrices of the form
		\[
		A=\begin{bmatrix}
			\lambda & 0 & 0 & 0\\
			0 & \mu & 0 & 0\\
			0 & 0 & \mu & 0\\
			0 & 0 & 0 & \mu
		\end{bmatrix}
		\]
		for any non-zero $\lambda$, $\mu$, hence forming a subgroup of $\PGL_4$ isomorphic to $\Gm$.
		
		Let $\Delta\subset\P^2$ be the locus of points $[a:b:c]\in\P^2$ whose associated cubic~\eqref{eq: S4 polynomial second form} is singular. What we have done above shows that the morphism
		\[
		\begin{tikzcd}
			\P^2\setminus\Delta\arrow[r] & \cM_{\fS_4}^{\P\W_4}
		\end{tikzcd}
		\]
		is a $\Gm$-torsor, where an element $A$ as above acts by
		\[
		A\cdot[a:b:c]=[\lambda^{-3}a,\lambda^{-1}\mu^{-2}b,\mu^{-3}c].
		\]
		Recall that $a$, $b$ and $c$ are different from zero. Therefore, by setting $c=1$, we can rewrite $\cM_{\fS_4}^{\P\W_4}$ as the quotient of an open subscheme of $\A^2$, with coordinates $(a,b)$, where $\lambda\in\Gm$ acts as
		\[
		\lambda\cdot(a,b)=(\lambda^{-3}a,\lambda^{-1}b).
		\]
		Again, as $b\not=0$, the quotient is isomorphic to an open subscheme of $\A^1\setminus\{0\}$, and for any $a\in\A^1\setminus\{0\}$ the corresponding surface has equation
		\[
		g(x,y,z,w)=ax^3+x(y^2+z^2+w^2-yw-zw)+w(y-z)(w-y-z).
		\]
		One then verifies that this is smooth if and only if $a\not=0$.
		
		As a byproduct, we have showed that $\cM_{\fS_4}^{\P\W_4}$ is irreducible and contains $\cM_{=\fS_4}$ as an open substack, hence it is contained in the closure of $\cM_{=\fS_4}$ in $\cM_{\fS_4}$. On the other hand, by Lemma~\ref{lem: cubic surface projective representation is locally constant} we know that $\cM_{\fS_4}^{\P\W_4}$ is open and closed in $\cM_{\fS_4}$, hence the previous containment is actually an equality.
		\end{proof}
		\begin{theorem}\label{thm: S4 monodromy}
			The field of definition of the lines over the generic surface $S_a$ over $\C(a)$ is $K:=\C\left(\sqrt{a},\sqrt{2a+1}\right)$. Moreover, $K/\C(a)$ is a Galois extension with Galois group $\mC_2\times\mC_2$. In particular,
			\begin{align*}
				\Mon_{\F^{=\fS_4}}=\Gal(K/\C(a))\simeq\mC_2\times\mC_2,\\
				\Mon_{\cF^{=\fS_4}}\simeq\fS_4\times\mC_2\times\mC_2.
			\end{align*}
		\end{theorem}
		\begin{proof}
			The first part is a straightforward computation, that we outline now. It is clear that three lines are $L_1=\{x=w=0\}$, $L_2=\{x=y-z=0\}$, $L_3=\{x=w-y-z=0\}$, and that these contain all the six chosen Eckardt points $p_1,\ldots,p_6$. Moreover, $\fS_4$ permutes these three lines, hence the field of definition of the set of lines intersecting $L_i$ does not depend on $i$ if this is a Galois extension of $\C(a)$. Then, one can substitute $w=ux$ for a parameter $u$, and impose that the quadric
			\[
			ax^2+(y^2+z^2+u^2x^2-uxy-uxz)+u(y-z)(ux-y-z)
			\]
			obtained by substitution and `dividing by $x$' is singular.
			Imposing that this factors, we get four possible values for $u$:
			\begin{align*}
				u=1,\qquad u=-1,\qquad u=i\sqrt{2}\sqrt{a},\qquad u=-i\sqrt{2}\sqrt{a}.
			\end{align*}
			The first two correspond to the tritangent Eckardt planes passing through $p_2$ and $p_5$, respectively. The last two correspond to the other two tritangent planes containing $L_1$; call these two planes by $D_2$ and $D_5$. In particular, we need to add $\sqrt{a}$ to be able to solve for the lines.
			
			As $\fS_4$ permutes the two Eckardt points $p_2$ and $p_5$, and it is easy to check that it also permutes two non-Eckardt planes $D_2$ and $D_5$. In particular, it is enough to solve for the lines for $u=1$ and $u=i\sqrt{2}\sqrt{a}$. In the first case, we want to split the equation
			\[
			(a+1)x^2-2xz+2z^2
			\]
			hence we need to add $\sqrt{2a+1}$, which suffice for the second case as well. Indeed, for $u=i\sqrt{2a}$ the quadric splits as
			\begin{align*}
				&(i\sqrt{a}x+(i\sqrt{a}-\sqrt{2}/2+i\sqrt{2}\sqrt{2a+1}/2)y+(b-\sqrt{2}/2+i\sqrt{2}\sqrt{2a+1}/2)z)\\
				\cdot&(i\sqrt{a}x+(i\sqrt{a}-\sqrt{2}/2-i\sqrt{2}\sqrt{2a+1}/2)y+(b-\sqrt{2}/2-i\sqrt{2}\sqrt{2a+1}/2)z).
			\end{align*}
			Therefore, the lines are defined after the field extension $\C(\sqrt{a},\sqrt{2a+1})$, whose Galois group (over $\C(a)$) is clearly $\mC_2\times\mC_2$. Therefore, $\Mon_{\cF_{=\fS_4}}\simeq\mC_2\times\mC_2$.
			
			By Corollary~\ref{cor: generic gerbe or isom complete case} and Corollary~\ref{cor: inclusion in normalizer}, we have $\Mon_{\F^{=\fS_4}}=\Mon_{\cF_{=\fS_4}}$, and $\Mon_{\cF^{=\fS_4}}\simeq\fS_4\times\Mon_{\F^{=\fS_4}}$. This concludes.
		\end{proof}
		\begin{remark}\label{rmk: difference S4}
			In this remark we explain the difference between the monodromy groups of $\cF^{=\fS_4}$ and $\F^{=\fS_4}$. Let $Z^{\fS_4}\subset\P^{19}$ be the locally-closed subscheme that parametrizes smooth embedded cubic surfaces with automorphism group containing $\fS_4$, and $Z^{=\fS_4}$ its open subscheme where the automorphism group coincides with $\fS_4$; they are both irreducible by Remark~\ref{rmk: irreducibility of MG}. Let $I\rightarrow\P^{19}$ be the cover given by the incident variety of the 27 lines. Since $Z^{=\fS_4}\rightarrow\cM^{=\fS_4}$ is a $\PGL_4$-torsor, the monodromy of the induced cover $f:I^{=\fS_4}\rightarrow Z^{=\fS_4}$ is equal to $\Mon_{\cF^{=\fS_4}}\simeq\fS_4\times(\mC_2)^2$, by Corollary~\ref{cor: invariance monodromy pullback flat maps with connected fibers}. Then, $\fS_4$ records the difference between $f$ and the cover over the coarse moduli space of non-embedded cubic surfaces $\L^{=\fS_4}\rightarrow\M^{=\fS_4}$. In the second case we saw how we can assume the cubic surface to have equation~\eqref{eq: equation S4 cubic surface}, with fixed Eckardt points as in Theorem~\ref{thm: S4 presentation}. On the other hand, the equation of a general embedded $\fS_4$-symmetric cubic surface is more complicated, and the monodromy of $\Mon_{f}=\Mon_{\cF^{=\fS_4}}\simeq\fS_4\times(\mC_2)^2$ acts non-trivially on the set of six Eckardt points. One can show that the monodromy group of the cover associated to these 6 Eckardt points has indeed Galois group $\fS_4$; since giving the $\fS_4$-action is the same as fixing the 6 Eckardt points by Remark~\ref{rmk: geometry S4 surfaces}, this explains the extra factor $\fS_4$. Notice that in both cases the monodromy group is solvable, and encodes the Galois group of the field extension needed to define the lines in the two settings.
		\end{remark}
		\subsection{$\fS_3$-Cubic Surfaces}\label{subsec: S3 cubic surfaces}
		In this section we compute the monodromy of the étale cover $\cF^{=\fS_3}:\cL^{=\fS_3}\rightarrow\cM^{=\fS_3}$; again, $\cM^{=\fS_3}$ is irreducible by Remark~\ref{rmk: irreducibility of MG}. Moreover, $\cM_{=\fS_3}\simeq\M^{=\fS_3}$ by Corollary~\ref{cor: generic gerbe or isom complete case}, as $\fS_3$ is complete, thus it is irreducible. By the same argument of Subsection~\ref{subsec: S4 cubic surfaces}, it is enough to compute the monodromy of $\F^{=\fS_3}:\L^{=\fS_3}\rightarrow\M^{=\fS_3}$, which satisfies $\Mon_{\F^{=\fS_3}}\simeq\Mon_{\cF_{=\fS_3}}$ by the above isomorphism.
		
		To compute $\Mon_{\cF_{=\fS_3}}$, we could use the same strategy as Subsection~\ref{subsec: S4 cubic surfaces}, that is finding a presentation for $\cM_{=\fS_3}$ and then solve for the lines. As we will see that the monodromy group is solvable, being isomorphic to $\fS_3\times\fS_3$, this method would certainly work. However, as the group is bigger, this would be an harder challenge. Therefore, we use a different method, which consists in first finding geometric constraints that the Galois group needs to satisfy, then use the hidden loops coming from the stabilizers to produce elements of the monodromy group.
		We start by understanding the closure of $\cM_{=\fS_3}$ in $\cM_{\fS_3}$.
		\begin{lemma}\label{lem: description of S3 closure}
			Set $\W_3=\mathds{1}\oplus\mathds{1}\oplus\V_2$, where $\V_2$ is the 2-dimensional standard $\fS_3$-representation. Then, $\cM_{\fS_3}^{\P\W_3}$ contains the closure of $\cM_{=\fS_3}$ in $\cM_{\fS_3}$.
			
			More precisely, every $\fS_3$-action arising as the automorphism subgroup generated by the involutions of three aligned Eckardt point has associated projective representation $\P\W_3$, and vice versa. The line $\P\V_2$ is the line passing through the three Eckardt points.
		\end{lemma}
		\begin{proof}
			By Lemma~\ref{lem: cubic surface projective representation is locally constant}, it is enough to show that $\cM_{\fS_3}^{\P\W_3}$ contains $\cM_{=\fS_3}$ in $\cM_{\fS_3}$. Therefore, it is enough to prove the second part of the statement.
			As usual, we can work with smooth cubic surfaces embedded in $\P^3$, and our goal is to compute the extended $\fS_3$-action on $\P^3$.
			Suppose that the $\fS_3$-action is induced by three Eckardt points, and call $M$ the line passing through them. Then, the corresponding affine cone $C_M$, which is a plane, gives a subrepresentation isomorphic to $\mathrm{V}_2$. Let $H_1$ and $H_2$ be the two hyperplanes consisting of fixed points of two fixed transposition in $\mathrm{S}_3$. Then, their intersection is a line $L$, whose points are fixed under the $\mathrm{S}_3$-action. It is easy to check that $L\cap M=\emptyset$. Let $C_L$ be the affine cone over $L$, hence a plane; then, the restriction of the action to such plane is the same on every line, i.e. $\mathds{1}\oplus\mathds{1}$ or $\mathrm{V}_1\oplus\mathrm{V}_1$, where $\V_1$ is the sign representation. We are using the fact that $\V_2\otimes\V_1\simeq\V_2$.
			
			For the converse, suppose that the associated projective representation is $\P\W_3$. Then, the three elements $\sigma_1$, $\sigma_2$, $\sigma_3$ of order 2 are involutions of $\P^3$, with fixed point locus of $\sigma_i$ equal to the disjoint union of a plane and a point $p_i$. It follows that $\sigma_i$ is induced by the involution associated to the Eckardt point $p_i$. This concludes.
		\end{proof}
		\begin{definition}
			We denote by $\cF_{\fS_3}^{\P\W_3}:\cL_{\fS_3}^{\P\W_3}\rightarrow\cM_{\fS_3}^{\P\W_3}$ the restriction of $\cF_{\fS_3}$ over $\cM_{\fS_3}^{\P\W_3}$.
		\end{definition}
		Now, we find geometric constraints for the Galois group, which are equivalent to those of Corollary~\ref{cor: inclusion in normalizer}. We avoid using the corollary to give an alternative way of finding the constraints without necessarily knowing the mondromy group of $\cF$. Since the monodromy needs to preserve the $\fS_3$-action, we get that the Galois group of $\cF_{\fS_3}^{\P\W_3}$ fixes the three Eckardt points, and preserves the $\fS_3$-orbits of the lines. In Figure~\ref{fig:3EckardtPoints} we describe the configuration of lines on a generic $\fS_3$-surface.
		\begin{figure}
			\centering
			\includegraphics[width=0.6\textwidth]{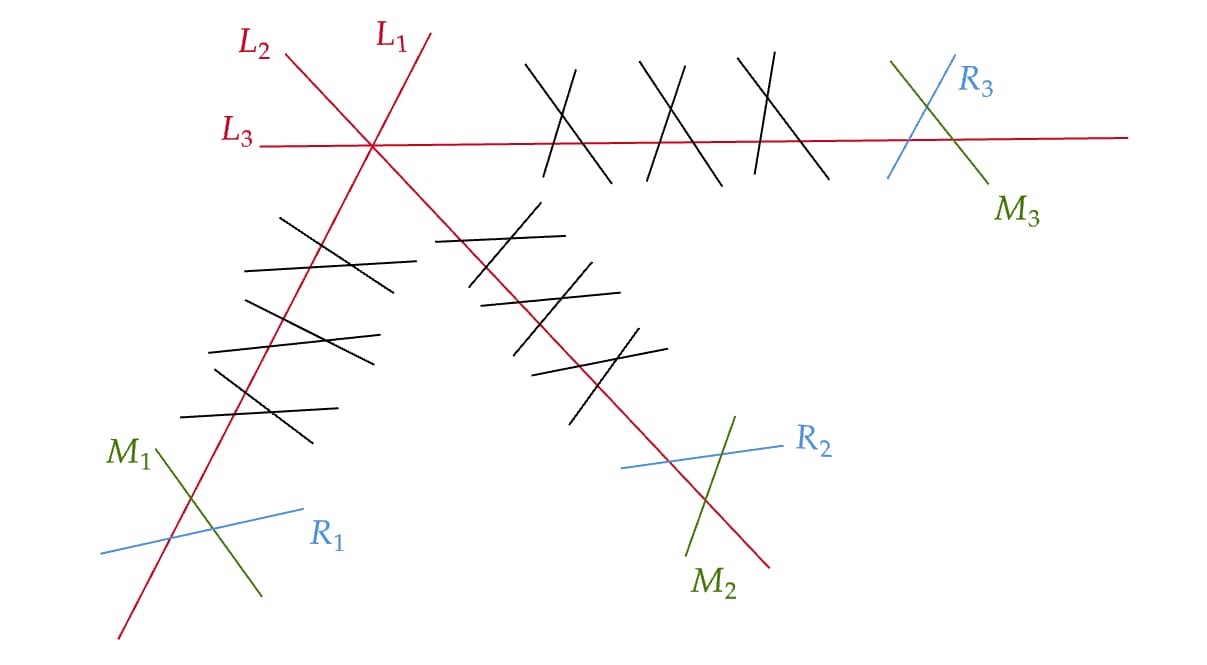} 
			\caption{Configuration of lines for $\fS_3$-cubic surface}
			\label{fig:3EckardtPoints}
		\end{figure}
		\begin{lemma}\label{lem: S3 monodromy constraint}
			The monodromy group of $\cF_{\fS_3}^{\P\W_3}$ is a subgroup of $\fS_3\times\fS_3$.
		\end{lemma}
		\begin{proof}
			We use the notation of Figure~\ref{fig:3EckardtPoints}; there, $L_i$, $M_i$ and $R_i$ are the Eckardt lines. Let $\H\subset\W(E_6)$ be the subgroup that fixes each of the three Eckardt points. Then, there is a normal subgroup $K\subset\H$ fixing each of $L_1$, $L_2$ and $L_3$, hence every other Eckardt line, and $\H/\K\simeq\fS_3$. As $\H$ preserves the intersection form, we have that $K\simeq\fS_3$. Since there is a section to the projection $\H\rightarrow\H/\K$, we get that $\H\simeq\fS_3\rtimes\fS_3$. By explicitly defining the elements of $\H$ one can check that the above is a direct product. Alternatively, this last part will follow from Theorem~\ref{thm: S3 presentation}.
		\end{proof}
		The second step is to construct monodromy elements. To do so, we find a presentation of $\cM_{\fS_3}^{\P\W_3}$, in particular computing the stabilizer groups of each of its point; however, the presentation is not strictly needed.
		\begin{theorem}\label{thm: S3 presentation}
			We have an isomorphism
			\[
			\cM_{\fS_3}^{\P\W_3}\simeq\left[\frac{\A^2\setminus\Delta}{\H}\right]
			\]
			where $\Delta$ is a finite subset and $\H\simeq\fS_3\times\mC_3$ acts on $\A^2$ by $\H\hookrightarrow\GL_2$ with image the subgroup generated by matrices $\begin{bmatrix}
				0 & 1\\
				1 & 0
			\end{bmatrix}$ and $\begin{bmatrix}
				\zeta_3^i & 0\\
				0 & \zeta_3^j
			\end{bmatrix}$, for every $i,j$.
			
			Every $(a,b)\in\A^2\setminus\Delta$ corresponds to the cubic surface
			\[
			x^3+y^3+z^3+w^3+(ax+by)zw
			\]
			and the $\fS_3$-action is obtained by restriction from $\fS_3\hookrightarrow\PGL_4$ given by
			\[
			(12)=\begin{bmatrix}
				1 & 0 & 0 & 0\\
				0 & 1 & 0 & 0\\
				0 & 0 & 0 & 1\\
				0 & 0 & 1 & 0
			\end{bmatrix},\quad (13)=\begin{bmatrix}
				1 & 0 & 0 & 0\\
				0 & 1 & 0 & 0\\
				0 & 0 & 0 & \zeta_3\\
				0 & 0 & \zeta_3^2 & 0
			\end{bmatrix}.
			\]
			Moreover, the Eckardt points inducing the action are
			\[
			p_1=[0:0:1:-1],\qquad p_2=[0:0:1:-\zeta_3^2],\qquad p_3=[0:0:1:-\zeta_3].
			\]
		\end{theorem}
		\begin{proof}
			Let $(f:S\rightarrow B,\rho:\fS_3\hookrightarrow\Aut(f))$ be an element of $\cM_{\fS_3}^{\P\W_3}$. As usual, we work Zariski-locally on $B$, so that we can assume $B=\Spec R$ and that $S\subset\P_R^3$ embeds $\fS_3$-equivariantly. In particular, every automorphism of $S$ extends to an automorphism of $\P^3:=\P^3_R$ preserving $S$.
			
			By Lemma~\ref{lem: description of S3 closure} we know that the $\fS_3$-action comes from 3 aligned Eckardt points, and that there is a line of fixed points not intersecting the line the Eckardt points are lying on. After a change of coordinates, we can assume that the Eckardt points are $p_1$, $p_2$ and $p_3$ as in the statement, so that $(12)$ is the involution associated to $p_1$ and $(13)$ the one associated to $p_2$. Moreover, we can assume that the line of points fixed by $\fS_3$ has equation $z=w=0$. Then $(12)$ fixes $p_1$, $[1:0:0:0]$, and $[0:0:0:1]$, while it swaps $p_2$ with $p_3$; imposing this and that the fix locus is the disjoint union of a plane and $p_1$, we get that $(12)$ has the form of the statement. An analogous argument works for $(13)$.
			
			Now we find the polynomial $g(x,y,z,w)$ defining $S$. Imposing that it is $\fS_3$-invariant, in a similar way as in~\ref{thm: S4 presentation}, we get that
			\begin{equation}\label{eq: S3 polynomial first form}
				g(x,y,z,w)=p(x,y)+(ax+by)zw+c(z^3+w^3).
			\end{equation}
			Now we find the $\fS_3$-equivariant morphisms between families of $\fS_3$-surfaces equivariantly embedded in $\P^3$ as above. As usual, such a morphism extends to an automorphism of $\P^3$ commuting with $\fS_3$. It is possible to show that these form the subgroup of $\PGL_4$ isomorphic to $\GL_2$ consisting of matrices
			\[
			\begin{bmatrix}
				A & 0\\
				0 & \mathrm{I}
			\end{bmatrix}
			\]
			for any $A\in\GL_2$, where $\mathrm{I}$ denotes the two-times-two identity matrix.
			
			We want to use this morphisms to simplify the form of the polynomial $g$ in equation~\eqref{eq: S3 polynomial first form}. Notice that, being $g$ smooth, $p$ is necessarily smooth as a polynomial in two variables, and $c\not=0$. Since $\GL_2$ acts transitively on the open locus of $\H^0(\cO_{\P^1}(3))$ parametrizing smooth binary forms of degree 3, after an $\fS_3$-equivariant change of coordinates we can assume that
			\begin{equation}\label{eq: S3 polynomial second form}
				g(x,y,z,w)=x^3+y^3+z^3+w^3+(ax+by)zw,
			\end{equation}
			where we have also used the invariance under scaling to set $c=1$. Notice that this is the form in the statement. Now, every isomorphism between two surfaces defined by polynomials as above is of the form $\begin{bmatrix}
				A & 0\\
				0 & \mathrm{I}
			\end{bmatrix}$, where $A$ is a two times two matrix as in the statement.
			
			Let $\Delta\subset\A^2$ be the locus parametrizing singular polynomials of the form~\eqref{eq: S3 polynomial second form}. A simple calculation shows that $\Delta$ is of codimension 2. Putting everything together, the above shows that
			\[
			\begin{tikzcd}
				\A^2\setminus\Delta\arrow[r] & \cM_{\fS_3}^{\P\W_3}
			\end{tikzcd}
			\]
			is an $\H$-torsor, with $\H$ as in the statement. This concludes.
		\end{proof}
		\begin{remark}\label{rmk: S3 automorphisms}
			As a consequence of Theorem~\ref{thm: S3 presentation}, it is simple to characterize the $\fS_3$-surfaces that have automorphisms: they are exactly those for which $a=b$. If $a=b\not=0$, then the stabilizer group is isomorphic to $\mC_2$, generated by
			\[
			\begin{bmatrix}
				0 & 1 & 0 & 0\\
				1 & 0 & 0 & 0\\
				0 & 0 & 1 & 0\\
				0 & 0 & 0 & 1
			\end{bmatrix},
			\]
			while the stabilizer group is isomorphic to the whole $\H$ if $a=b=0$. 
		\end{remark}
		Thanks to Lemma~\ref{lem: S3 monodromy constraint}, we know that the monodromy group of $\cF_{\fS_3}^{\P\W_3}$ is contained in $\fS_3\times\fS_3$, and we claim that the two groups are actually equal.
		
		By Remark~\ref{rmk: S3 automorphisms}, the Fermat cubic surface $S_{\text{Fer}}$ with polynomial $x^3+y^3+z^3+w^3$ and $\fS_3$-action as in Theorem~\ref{thm: S3 presentation} has automorphism group $\H\simeq\fS_3\times\mC_3$. Since every automorphism of a smooth cubic surface acts faithfully on the lines, we get that $\H$ injects into the monodromy of $\cF_{\fS_3}^{\P\W_3}$.
		Therefore, the monodromy group of $\cF_{\fS_3}^{\P\W_3}$ is squeezed between the groups $\fS_3\times\mC_3$ and $\fS_3\times\fS_3$, hence it is enough to produce another element of degree 2. Recall that for any injective morphism of groups $\psi:\G_1\hookrightarrow\G_2$ there exists an induced representable morphism $\Phi_{\psi}:\cM_{\G_2}\rightarrow\cM_{\G_1}$, which restricts the $\G_2$-action to the $\G_1$-action under $\psi$; see Definition~\ref{def: action of AutG}. In particular, the inclusion $\psi:\fS_3\hookrightarrow\fS_4$ as the stabilizer of the fourth element induces a morphism
		\[
		\begin{tikzcd}
			\Phi_{\psi}:\cM_{\fS_4}\arrow[r] & \cM_{\fS_3}.
		\end{tikzcd}
		\]
		Notice that the restriction of the projective $\fS_4$-representation $\P\W_4$ to $\fS_3$ is isomorphic to $\P\W_3$, hence we obtain a morphism
		\[
		\begin{tikzcd}
			\Phi_{\psi}^{\P\W}:\cM_{\fS_4}^{\P\W_4}\arrow[r] & \cM_{\fS_3}^{\P\W_3},
		\end{tikzcd}
		\]
		showing that the $\fS_4$-monodromy injects into the $\fS_3$-monodromy. As the first is isomorphic to $\mC_2\times\mC_2$ by Theorem~\ref{thm: S4 monodromy}, this gives the elements of degree 2 that we needed. Putting everything together, we get the following result.
		\begin{theorem}\label{thm: S3 monodromy}
			We have
			\begin{align*}
				&\Mon_{\F^{=\fS_3}}\simeq\fS_3\times\fS_3,\\ &\Mon_{\F_{=\fS_3}}\simeq\fS_3\times\fS_3\times\fS_3.
			\end{align*} 
		\end{theorem}
		\begin{proof}
			The first isomorphism follows by the above discussion. The second isomorphism follows from the first and the end of Corollary~\ref{cor: inclusion in normalizer}.
		\end{proof}
		\begin{remark}\label{rmk: difference S3}
			As in Remark~\ref{rmk: difference S4}, the difference between the two monodromy groups is explained by looking at the appropriate restriction of the covering $I\rightarrow\P^{19}$, with source the incident variety of lines on cubic surfaces. Then, the extra factor $\fS_3$ is the Galois group of the cover induced by the three Eckardt points, that are instead fixed in $\cM_{=\fS_3}\simeq\M^{=\fS_3}$.
		\end{remark}
		\subsection{$\fS_3\times\mC_2$-Cubic Surfaces}\label{subsec: S3xmu2 cubic surfaces}
		In this subsection we deal with the case of $\fS_3\times\mC_2$-cubic surfaces, which is different from the previous cases in the following precise sense. When we focused on $\G$-surfaces with $\G=\fS_4$ or $\G=\fS_3$, we leveraged the fact that $\G$ was complete to conclude that $\cM_{=\G}\rightarrow\M^{=\G}$ was an isomorphism; for $\G=\fS_3\times\mC_2$ this is no longer true. In this case, we have $\Zen(\fS_3\times\mC_2)=\mC_2$ and $\Aut(\fS_3\times\mC_2)\simeq\Inn(\fS_3\times\mC_2)\times\mC_2\simeq\fS_3\times\mC_2$, where the outer automorphism swaps $(\sigma,0)$ with $(\sigma,1)$ for every transposition $\sigma\in\fS_3$. Then, Corollary~\ref{cor: generic gerbe or isom complete case} tells us that
		\[
		\begin{tikzcd}
			\pi^{=\fS_3\times\mC_2}\circ\Phi_{=\fS_3\times\mC_2}:\cM_{=\fS_3\times\mC_2}\arrow[r] & \M^{=\fS_3\times\mC_2}
		\end{tikzcd}
		\]
		is the composite of a banded $\mC_2$-gerbe and a $\mC_2$-torsor. It turns out that these are both trivial; in particular, $\cM_{=\G}$ is not always connected. This is the content of the next lemma.
		\begin{remark}\label{rmk: S3xmu2 action from Eckardt points}
			We say that 4 Eckardt points in a cubic surface $S$ form an $\fS_3\times\mC_2$-configuration if the subgroup of $\Aut(S)$ generated by the Eckardt involutions is isomorphic to $\fS_3\times\mC_2$. In this case, three of those points $p_1$, $p_2$ and $p_3$ are aligned, and each lies on one Eckardt line passing through the fourth point $p$.
		\end{remark}	 
		\begin{lemma}\label{lem: description of S3xmu2 closure}
			Let $\W_2=\mathds{1}\oplus\mathds{-1}\oplus\V_2$ be the $\fS_3\times\mC_2$-representation where $\V_2$ is the standard 2-dimensional $\fS_3$-representation and $\mC_2$ acts trivially on it, while $\mathds{-1}$ is the representation with $\fS_3$ acting trivially and $\mC_2$ acting by multiplication. Let $\widetilde{\W}_2=\mathds{1}\oplus-\V_1\oplus\V_2$, where $-\V_1=-\mathds{1}\otimes\V_1$. Then, $\cM_{=\fS_3\times\mC_2}$ is disconnected, and its closure in $\cM_{\fS_3\times\mC_2}$ is contained in the disjoint union of $\cM_{\fS_3\times\mC_2}^{\P\W_2}$ and $\cM_{\fS_3\times\mC_2}^{\P\widetilde{\W_2}}$. Moreover, these two components are isomorphic and swapped under the action of $\Out(\fS_3\times\mC_2)\simeq\mC_2$.
		\end{lemma}
		\begin{proof}
			As usual, we can work with smooth cubic surfaces embedded in $\P^3$, and our goal is to compute the extended $\fS_3\times\mC_2$-action on $\P^3$. Set $\G=\fS_3\times\mC_2$, and suppose first that the $\G$-action is induced by an $\fS_3\times\mC_2$-Eckardt configuration. Let $p_1$, $p_2$, $p_3$ and $p$ be as in Remark~\ref{rmk: S3xmu2 action from Eckardt points}, such that $(12)\in\fS_3$ is induced by $p_1$ and $(13)\in\fS_3$ by $p_2$; call $\iota$ the involution associated to $p$. Then, the line $M$ passing through the three aligned Eckardt points is preserved by $\G$, and $\iota$ acts trivially on it, as it fixes the three Eckardt points. By the proof of Lemma~\ref{lem: description of S3 closure}, we know that $\fS_3$ acts trivially on a line $L$ disjoint from $M$. As $\iota$ commutes with $\fS_3$, it preserves $L$; in particular, the $\fS_3\times\mC_2$-representation inducing the action on $\P^3$ splits as the direct sum of $\V_2$ and the cone $C_L\subset\C^{\oplus4}$ over $L$, which is a plane. As $\iota$ acts non-trivially on $L$, the induced $\mC_2$-representation on $C_L$ is $\mathds{-1}\oplus\mathds{1}$. As $\fS_3$ acts trivially on $L$, we finally get the projective representation $\P\W_2$. The converse can be proved as in Lemma~\ref{lem: description of S3 closure}. The only other possible $\G$-action on $S$ is the one obtained by first applying the outer automorphism of $\G$, and then taking the action induced by an $\fS_3\times\mC_2$-Eckardt configuration. Then, the action on $M$ is the same one, while the projective representation corresponding to the line $L$ is $\P(\mathds{1}\oplus-\V_1)$. This shows that $\cM_{=\fS_3\times\mC_2}$ is contained in the disjoint union of the two stacks of the statement. Then, Lemma~\ref{lem: cubic surface projective representation is locally constant} shows that the same holds for the closure.
		\end{proof}
		\begin{theorem}\label{thm: S3xmu2 presentation}
			We have an isomorphism
			\[
			\cM_{\fS_3\times\mC_2}^{\P\W_2}\simeq\left[\frac{\A^1\setminus\Delta}{\mC_6}\right]
			\]
			where
			\[
			\Delta=\left\{\left[1:1:-\frac{2a}{3}:-\frac{2a}{3}\right],\left[1:1:-\frac{2a\zeta_3}{3}:-\frac{2a\zeta_3^2}{3}\right],\left[1:1:-\frac{2a\zeta_3^2}{3}:-\frac{2a\zeta_3}{3}\right]\right\}
			\]
			and $\mC_6$ acts on $\A^1$ by $\mC_6\rightarrow\Gm$ sending $\lambda$ to $\lambda^2$.
			
			Every $a\in\A^1\setminus\Delta$ corresponds to the cubic surface $S_a$ with equation
			\[
			x^3+y^3+z^3+w^3+a(x+y)zw
			\]
			and the $\fS_3\times\mC_2$-action is obtained by restriction from $\fS_3\times\mC_2\hookrightarrow\PGL_4$ given by
			\[
			(12)=\begin{bmatrix}
				1 & 0 & 0 & 0\\
				0 & 1 & 0 & 0\\
				0 & 0 & 0 & 1\\
				0 & 0 & 1 & 0
			\end{bmatrix},\quad (13)=\begin{bmatrix}
				1 & 0 & 0 & 0\\
				0 & 1 & 0 & 0\\
				0 & 0 & 0 & \zeta_3\\
				0 & 0 & \zeta_3^2 & 0
			\end{bmatrix},\quad\iota=\begin{bmatrix}
				0 & 1 & 0 & 0\\
				1 & 0 & 0 & 0\\
				0 & 0 & 1 & 0\\
				0 & 0 & 0 & 1
			\end{bmatrix}
			\]
			where $\iota$ is the generator of $\mC_2$. The Eckardt points inducing the action are
			\begin{align*}
				p_1&=[0:0:1:-1], && p_2=[0:0:1:-\zeta_3^2],\\
				\quad p_3&=[0:0:1:-\zeta_3], && p_4=[1:-1:0:0].
			\end{align*}
		\end{theorem}
		\begin{proof}
			Let $(f:S\rightarrow B,\rho:\fS_3\times\mC_2\hookrightarrow\Aut(f))$ be an element of $\cM_{\fS_3\times\mC_2}^{\P\W_2}$. As usual, we work Zariski-locally on $B$, so that we can assume $B=\Spec R$ and $S\subset\P^3$ embedded $\fS_3\times\mC_2$-equivariantly.
			
			By Lemma~\ref{lem: description of S3xmu2 closure} we know that the action of $\fS_3\times\mC_2$ is induced by a $\fS_3\times\mC_2$-Eckardt configuration. As in Theorem~\ref{thm: S3 presentation}, we can apply a change of coordinates so that the three aligned Eckardt points $p_1$, $p_2$, $p_3$ are as in the statement, and that the line fixed by $\fS_3$ has equation $z=w=0$. In particular, the $\fS_3$-action is the one given in the statement. Now, notice that the fourth Eckardt point of the $\fS_3\times\mC_2$-configuration is fixed by $\fS_3$, hence it lies on the line $z=w=0$. After apply a change of coordinates affecting only the coordinates $x$ and $y$, hence commuting with $\fS_3$, we can assume that $p_4=[1:-1:0:0]$. To compute $\iota$, first notice that $\fS_3\times\mC_2$ fixes two distinct points, one of which is $p_4$. The other fixed point $q$ still necessarily lies on the line $z=w=0$, thus we can apply another change of basis commuting with $\fS_3$ and fixing $p_4$ so that $q=[1:1:0:0]$. All this choices uniquely determine $\iota$; to find it, one imposes that it commutes with $\fS_3$, and that the locus of fixed points consists in the disjoint union of $p_4$ and a plane passing through $q$.
			
			Now, we find the polynomials $g$ invariant under the $\fS_3\times\mC_2$-action. As the $\fS_3$-action is the same as the one in Theorem~\ref{thm: S3 presentation}, we get that $g$ has necessarily the form
			\[
			g(x,y,z,w)=x^3+t^3+z^3+w^3+(ax+by)zw.
			\]
			Imposing the invariance under $\mC_2$, which simply swaps $x$ and $y$, we get that $a=b$, that is
			\begin{equation}\label{eq: S3xmu2 polynomial}
				g(x,y,z,w)=x^3+t^3+z^3+w^3+a(x+y)zw.
			\end{equation}
			A simple computation shows that $g$ is singular if and only if $a\in\Delta$, where $\Delta\subset\A^1$ is as in the statement. The same argument as in Theorems~\ref{thm: S4 presentation} and~\ref{thm: S3 presentation} shows that every morphism between families of $\fS_3\times\mC_2$-cubic surfaces are obtained locally on the base by restriction of elements of $\PGL_4$ commuting with the action of $\fS_3\times\mC_2$ on $\P^3$ described above. These matrices form the subgroup as in the statement. Putting everything together, this proves the desired isomorphism.
		\end{proof}
		\begin{remark}
			Notice that $\cM_{=\fS_3\times\mC_2}^{\P\W_2}\rightarrow\M^{=\fS_3\times\mC_2}$ is a $\mC_2$-gerbe. In particular, Theorem~\ref{thm: S3xmu2 presentation} gives also a presentation of $\M^{=\fS_3\times\mC_2}$, simply by taking the quotient by $\mC_3$ instead of $\mC_6$, showing that the $\mC_2$-gerbe is trivial. In particular, $\Mon_{\cF_{\fS_3\times\mC_2}^{\P\W_2}}\simeq\mC_2\times\Mon_{\F^{=\fS_3\times\mC_2}}$. Notice that also the $\mC_2$-torsor appearing in the decomposition of $\pi^{=\fS_3\times\mC_2}\circ\Phi_{=\fS_3\times\mC_2}$ is trivial, by Lemma~\ref{lem: description of S3xmu2 closure}; compare with Corollary~\ref{cor: generic gerbe or isom complete case}. Therefore, we get a section to $\pi^{=\fS_3\times\mC_2}:\cM^{=\fS_3\times\mC_2}\rightarrow\M^{=\fS_3\times\mC_2}$, showing that the gerbe is trivial. As a consequence, we get
			\[
			\Mon_{\cF^{=\fS_3\times\mC_2}}\simeq\Mon_{\F^{=\fS_3\times\mC_2}}\times\fS_3\times\mC_2.
			\]
		\end{remark}
		With the notation of Theorem~\ref{thm: S3xmu2 presentation}, the only point of $a\in\A^1\setminus\Delta$ with stabilizer bigger than $\mC_2$ is $a=0$, corresponding to the Fermat cubic, that has stabilizer equal to $\mC_6$. This shows that $\mC_6$ injects into the monodromy of $\cF_{=\fS_3\times\mC_2}$.
		
		By the same argument as in Lemma~\ref{lem: S3 monodromy constraint}, the monodromy of the $\mC_2$-rigidification of $\cF_{\fS_3\times\mC_2}^{\P\W_2}$, which is equal to the monodromy of $\F^{=\fS_3\times\mC_2}$, is contained in $\fS_3\times\fS_3$. After imposing that the monodromy respects also the action of $\mC_2$, it is easy to check that the monodromy is a subgroup of $\fS_3\times\mC_2$. It is not clear whether the two groups are equal or not without explicitly computing the equations of the lines, which we do in the following theorem.
		\begin{theorem}\label{thm: S3xmu2 monodromy}
			The field of definition of the lines over the generic surfce $S_a$ over $\C(a)$ is $K:=\C\left(u,\sqrt{4a^3+27}\right)$, where $u\in\overline{\C(a)}$ satisfies $u^3+au^2+1=0$. Moreover, $K/\C(a)$ is a Galois extension with Galois group $\fS_3\times\mC_2$. In particular,
			\begin{align*}
				&\Mon_{\F^{=\fS_3\times\mC_2}}=\Gal(K/\C(a))\simeq\fS_3\times\mC_2,\\
				&\Mon_{\cF_{\fS_3\times\mC_2}^{\P\W_2}}\simeq\fS_3\times\mC_2\times\mC_2,\\
				&\Mon_{\cF^{=\fS_3\times\mC_2}}\simeq(\fS_3\times\mC_2)^2.
			\end{align*}
		\end{theorem}
		\begin{proof}
			By the above discussion, it is enough to prove the first statement. The proof is the same as for Theorem~\ref{thm: S4 monodromy}, hence we just sketch the computations needed to prove the result. First, we already know the equations of three lines:
			\begin{align*}
				L_1:\ x+y=z+w=0,\quad L_2:\ x+y=\zeta_3z+w=0,\quad L_3:\ x+y=\zeta_3^2z+w=0.
			\end{align*}
			Moreover, $\fS_3$ permutes the three lines, hence we need to compute the field of definition only for the lines intersecting one of them, say $L_1$. First we compute the tritangent planes through $L_1$; to do that, one substitutes $z+w=u(x+y)$ for a parameter $u$, then `divides' by $x+y$ and imposes that the resulting quadric splits as the sum of two lines. We get that $u=3a$, which corresponds to the plane passing through the Eckardt point $p_1$, or that $u$ satisfies the equation
			\begin{equation}\label{eq: equation satisfied by u S3xmu2 case}
				u^3+au^2+1=0.
			\end{equation}
			The Galois closure of the extension $\C(u)/\C$ has Galois group $\fS_3$, as the discriminant of the polynomial~\eqref{eq: equation satisfied by u S3xmu2 case} is $4a^3+27$, which is not a square. Then, one can check directly that the other lines are defined over $K=\C(u,\sqrt{4a^3+27})$.
		\end{proof}
		\begin{remark}\label{rmk: difference S3xmu2}
			Again, the extra $\fS_3$ factor is the Galois group of the covering induced by the three aligned Eckardt points. The extra $\mC_2$ is instead related to the choice of Eckardt lines that do not pass through $p$, as those are swapped by the Eckardt involution associated to $p$.
		\end{remark}
		\subsection{$\mC_2$-Cubic Surfaces}\label{subsec: mu2 cubic surfaces}
		We conclude the paper with the case of $\mC_2$-cubic surfaces, whose moduli space is 3-dimensional. By the discussion on Eckardt points, it follows that these correspond to the cubic surfaces that have at least one Eckardt point.
		
		By Proposition~\ref{prop: generic gerbe or isom}, we know that the first morphism in
		\[
		\begin{tikzcd}
			\cM_{=\mC_2}\arrow[r,"\Phi_{=\mC_2}"',"\simeq"] & \cM^{=\mC_2}\arrow[r,"\pi_{\cM^{=\mC_2}}"'] & \M^{=\mC_2}
		\end{tikzcd}
		\]
		is an isomorphism, while the second is a $\mC_2$-gerbe; in particular, they are all irreducible by Remark~\ref{rmk: irreducibility of MG}. By Example~\ref{exm: G cubic surfaces}, there exists an exact sequence
		\begin{equation}\label{eq: exact sequence gerbe mu2}
			\begin{tikzcd}
				0\arrow[r] & \mC_2\arrow[r] & \Mon_{\cF^{=\mC_2}}\arrow[r] & \Mon_{\F^{=\mC_2}}\arrow[r] & 0.
			\end{tikzcd}
		\end{equation}
		In particular, it is enough to compute $\Mon_{\cF^{=\mC_2}}=\Mon_{\cF_{=\mC_2}}$. As always, we start by understanding the closure of $\cM_{=\mC_2}$ in $\cM_{\mC_2}$.
		\begin{lemma}\label{lem: representation mu2}
			Let $\W_1=\mathds{1}^{\oplus3}\oplus-\mathds{1}$. Then, $\cM_{\mC_2}^{\P\W_1}$ is the closure of $\cM_{=\mC_2}$ in $\cM_{\mC_2}$.
		\end{lemma}
		\begin{proof}
			Every involution $\iota$ with induced projective representation $\P\W_1$ has fixed point locus given by the disjoint union of the point $\P(-\mathds{1})$ and the plane $\P(\mathds{1}^{\oplus3})$. Lemma~\ref{lem: facts on Eckardt points} shows that $\iota$ is an Eckardt involution, and that the converse holds. It follows immediately that the closure of $\cM_{=\mC_2}$ is contained in $\cM_{\mC_2}^{\P\W_1}$, and we are left with showing that the latter is connected. Let $\cM_{E}$ be the moduli stack of pairs $(S,e)$ where $e\in S$ is an Eckardt point in a smooth cubic surface; the above discussion easily implies that $\cM_E\simeq\cM_{\mC_2}^{\P\W_1}$. We claim that $\cM_E$ is connected. Indeed, every pair $(S,e)$ is connected to a pair $(S_{\mathrm{Fer}},\widetilde{e})$, where $S_{\mathrm{Fer}}$ is the Fermat cubic. Then, one concludes by noticing that $\Aut(S_{\mathrm{Fer}})$ acts transitively on its 18 Eckardt points.
		\end{proof}
		\begin{remark}\label{rmk: S3xmu2 maps to mu2}
			Notice that the above proof shows that $\cM_{\mC_2}^{\P\W_1}$ is isomorphic to the moduli stack parametrizing cubic surfaces with one marked Eckardt point. In particular, there is a morphism
			\[
			\begin{tikzcd}
				\cM^{=\fS_3\times\mC_2}\arrow[r] & \cM_{\mC_2}^{\P\W_1}
			\end{tikzcd}
			\]
			sending a smooth cubic surface $S$ with $\Aut(S)$ to the pairs $(S,e)$ where $e\in S$ is the only Eckardt point whose induced involution lies in the center of $\fS_3\times\mC_2$. Therefore,
			\[
			\begin{tikzcd}
				(\fS_3\times\mC_2)^2\simeq\Mon_{\cF^{=\fS_3\times\mC_2}}\arrow[r,hookrightarrow] & \Mon_{\cF_{\mC_2}^{\P\W_1}}\simeq\Mon_{\cF^{=\mC_2}}.
			\end{tikzcd}
			\]
		\end{remark}
		
		We denote by $\mathrm{GO}_4^+(3)$ the subgroup of $\GL_4(\mathbb{F}_3)$ consisting of elements that preserve the non-singular quadratic form on $\mathbb{F}_3^{\oplus4}$ with Witt index 2, where $\mathbb{F}_3$ is the field with three elements. It has center isomorphic to $\mC_2$, whose quotient is denoted by $\mathrm{PGO}_4^+(3)$; we have that
		\[
		\mathrm{PGO}_4^+(3)\simeq(A_4\times A_4)\rtimes(\mC_2)^2\simeq(\mC_2)^4\rtimes(\fS_3\times\fS_3),
		\]
		with code `SmallGroup(576,8654)' in~\cite{GAP4}.
		\begin{remark}\label{rmk: monodromy constraint mu2}
			Since $\Mon_{\cF^{=\mC_2}}\simeq\Mon_{\cF_{=\mC_2}}$, this group preserves the $\mC_2$-action, hence it fixes the Eckardt point and preserves the associated tritangent plane. By~\cite[Section 9.5.2]{Dol12}, the stabilizer of a tritangent plane in $W(E_6)$ is isomorphic to $\mathrm{GO}_4^+(3)$. By Example~\ref{exm: G cubic surfaces}, it follows that $\Mon_{\F^{=\mC_2}}$ is contained in $\mathrm{GO}_4^+(3)/\mC_2\simeq\mathrm{PGO}_4^+(3)$, as $\mC_2$ coincides with the center of $\mathrm{GO}_4^+(3)$. Alternatively, one could have used Corollary~\ref{cor: inclusion in normalizer}.
		\end{remark}
		\begin{theorem}\label{thm: mu2 monodromy}
			We have isomorphisms
			\begin{align*}
				&\Mon_{\F^{=\mC_2}}\simeq\mathrm{PGO}_4^+(3),\\
				&\Mon_{\cF^{=\mC_2}}\simeq\Mon_{\cF_{\mC_2}^{\P\W_1}}\simeq\mathrm{GO}_4^+(3).
			\end{align*}
		\end{theorem}
		\begin{proof}
			By Remark~\ref{rmk: S3xmu2 maps to mu2}, we know that $\Mon_{\cF^{=\mC_2}}$ contains a group $\K$ of order 144 with no elements of order 4. In particular, this group has index dividing 8 in $\mathrm{GO}_4^+(3)$. Recall that there exist a surface with exactly one Eckardt point and automorphism group isomorphic to $\mC_8$, see~\cite[Table 5]{BLP23}, which corrects a mistake in~\cite[Table 9.6]{Dol12}. By the usual argument, this yields an element $\sigma$ of order 8 in $\Mon_{\cF^{=\mC_2}}$. If $\H$ is the group generated by this element and $\K$, then $\H$ has index dividing 4 in $\mathrm{GO}_4^+(3)$; we claim that it is smaller than 4. If not, $\K$ would be of index 2 in $\H$, hence normal with quotient isomorphic to $\mC_2$. This implies that $\sigma^2\in\K$, which contradicts the fact that $\H$ does not contain elements of order 4. To summarize, we have proved that $\Mon_{\cF^{=\mC_2}}$ is either $\mathrm{GO}_4^+(3)$ or of index 2, thus normal. To conclude it is enough to prove that $\Mon_{\F^{=\mC_2}}=\mathrm{PGO}_4^+(3)\simeq(\mC_2)^4\rtimes(\fS_3\times\fS_3)$. This follows from the fact that every normal subgroup of $\mathrm{PGO}_4^+(3)$ contains the socle $(\mC_2)^4$, and we know the other factor $\fS_3\times\fS_3$ to be contained in $\Mon_{\F^{=\mC_2}}$.
		\end{proof}
		\begin{remark}
			Notice that the central extension
			\[
			\begin{tikzcd}
				0\arrow[r] & \mC_2\arrow[r] & \mathrm{GO}_4^+(3)\arrow[r] & \mathrm{PGO}_4^+(3)\arrow[r] & 0
			\end{tikzcd}
			\]
			is non-split, as $\mathrm{PGO}_4^+(3)$ does not contain any element of order 8, while $\mathrm{GO}_4^+(3)$ does. This implies that the $\mC_2$-gerbe $\cM^{=\mC_2}\rightarrow\M^{=\mC_2}$ is non-trivial.
		\end{remark}
		
		\bibliographystyle{amsalpha}
		\bibliography{library}
		
	\end{document}